\documentclass[11pt,leqno,oneside]{amsart}
\usepackage{hyperref,amsmath,amsfonts,amssymb,amsthm,enumerate,mathtools}
\usepackage[a4paper, left=2cm, right=2cm, top=3cm, bottom=2cm]{geometry}
\usepackage[numbers,sort&compress]{natbib}
\def\R{\mathbb{R}}
\def\rn{\mathbb{R}^n}
\def\N{\mathbb{N}}
\def\Z{\mathbb{Z}}
\def\d{{\fam0 d}}
\def\dist{\operatorname{dist}}
\let\hat\widehat
\input widebar
\let\bar\widebar

\newtheoremstyle{MyPlain}{}{}{\itshape}{}{\bfseries}{.}{5pt plus 4pt minus 3pt}{\thmname{#1}\thmnumber{ #2}\thmnote{ \textbf{[#3]}}}
\theoremstyle{MyPlain}
\newtheorem{theorem}{Theorem}[section]
\newtheorem{lemma}[theorem]{Lemma}
\newtheorem{proposition}[theorem]{Proposition}
\newtheorem{corollary}[theorem]{Corollary}
\newtheoremstyle{MyRemark}{}{}{\upshape}{}{\bfseries}{.}{5pt plus 1pt minus 1pt}{}
\theoremstyle{MyRemark}
\newtheorem{remark}[theorem]{Remark}
\newtheorem{example}[theorem]{Example}
\numberwithin{equation}{section}

\expandafter\let\expandafter\oldproof\csname\string\proof\endcsname
\let\oldendproof\endproof
\renewenvironment{proof}[1][\proofname]{%
  \oldproof[{{\bf #1.}}]%
}{\oldendproof}

\begin{document}

\title{Optimal  domain spaces in Orlicz-Sobolev  embeddings}

\thanks{%
This research was partly supported by the Research Project of Italian Ministry of University and
Research (MIUR)  2012TC7588 ``Elliptic and parabolic partial differential
equations: geometric aspects, related inequalities, and
applications" 2012,  by  GNAMPA  of Italian INdAM (National
Institute of High Mathematics),
by the grant P201-13-14743S of the Grant Agency of the Czech Republic and by the Charles University, project GAUK No.~33315.}

\begin{abstract}
We deal  with Orlicz-Sobolev  embeddings  in  open subsets of $\mathbb{R}^n$. A necessary and sufficient condition is established for the existence of an optimal, i.e.\ largest possible, Orlicz-Sobolev  space  continuously embedded into a given Orlicz space.   Moreover, the optimal Orlicz-Sobolev  space is  exhibited whenever it exists. Parallel questions are addressed for Orlicz-Sobolev embeddings into Orlicz spaces with respect to a Frostman measure, and, in particular, for trace embeddings  on the boundary. 
\end{abstract}

\author{Andrea Cianchi}
\address{Dipartimento di Matematica e Informatica ``Ulisse Dini'',
University of Florence,
Viale Morgagni 67/A, 50134
Firenze,
Italy}
\email{andrea.cianchi@unifi.it}
\urladdr{0000-0002-1198-8718}

\author{V\'\i t Musil}
\address{Department of Mathematical Analysis,
Faculty of Mathematics and Physics,
Charles University,
So\-ko\-lo\-vsk\'a~83,
186~75 Praha~8,
Czech Republic}
\email{musil@karlin.mff.cuni.cz}
\urladdr{0000-0001-6083-227X}

\date{\today}


\maketitle

\bibliographystyle{plain}

\section{Introduction}

The present paper deals with Orlicz-Sobolev embeddings, namely embeddings of Sobolev type, involving norms in Orlicz spaces. The family of Orlicz spaces includes that of the usual Lebesgue spaces, and provides a  flexible, well suited framework for a unified  description  of Sobolev embeddings. 
Orlicz-Sobolev spaces are  an appropriate functional setting for the analysis of nonlinear partial differential equations and variational problems governed by nonlinearities of non-necessarily polynomial type. The study of these problems has received an increasing attention over the years -- see e.g. \citep{AcerbiMingione, Ball, Baroni, BreitSchirra, BreitStrofVerde, BulDiSchw, BulMaMa, Eyring, Korolev, Lieberman, Marcellini, Talenti79, Talenti90, Wr} -- and is motivated, among other reasons, by applications to mathematical models for physical phenomena, such as nonlinear elasticity and non-Newtonian  fluid-mechanics.

A basic version of the Orlicz-Sobolev embeddings to be considered here amounts to
\begin{equation}\label{sob0}
W^{m,A}_0 (\Omega) \to L^B(\Omega)\,,
\end{equation}
where $\Omega$ is an open subset of Euclidean space $\rn$, $n \geq 2$, having Lebesgue measure $|\Omega|$, $A$ and $B$ are Young functions, $L^B(\Omega)$ is the Orlicz space on $\Omega$ built upon $B$, and $W^{m,A}_0 (\Omega)$ is the $m$-th order Orlicz-Sobolev space built upon $A$. The subscript $0$ denotes that functions vanishing (in a suitable sense) on the boundary $\partial \Omega$, together with their derivatives up to the order $m-1$, are taken into account. The arrow $\lq\lq \to "$ stands for continuous inclusion. Precise definitions on these topics are recalled in Section~\ref{back}.

We are concerned with the optimal form of the relevant embeddings. Given $A$, we say that $L^B(\Omega)$ is the optimal Orlicz target space  in \eqref{sob0} if it is the smallest Orlicz space on $\Omega$ that renders \eqref{sob0} true. The expression \lq\lq smallest" means that if \eqref{sob0} holds with $L^B(\Omega)$ replaced with another Orlicz space $L^{\widehat B}(\Omega)$, then $L^B(\Omega) \to L^{\widehat B}(\Omega)$. Analogously, given $B$, the space $W^{m,A}_0 (\Omega)$ is said to be the optimal Orlicz-Sobolev domain  in \eqref{sob0} if it is the largest Orlicz-Sobolev space on $\Omega$ for which  \eqref{sob0} holds. Namely, if, whenever \eqref{sob0} holds with $W^{m,A}_0 (\Omega)$ replaced by another Orlicz-Sobolev space $W^{m,\widehat A}_0 (\Omega)$, then $W^{m,\widehat A}_0 (\Omega) \to W^{m, A}_0 (\Omega)$.

The question of best possible Orlicz target spaces in Sobolev type embeddings has attracted the attention  of various authors over the years. In particular,  embeddings for  the  critical Sobolev space $W^{m, \frac nm}_0(\Omega)$, and for special Orlicz-Sobolev spaces \lq\lq close" to it, have been investigated in several contributions, including  \cite{Yudovich,  Pokhozhaev, Strichartz:Indiana, Trudinger, HMT, Mazya, Talenti,   EGO}.
Results for arbitrary Orlicz-Sobolev spaces, which  however need not provide the optimal Orlicz target, can be found in \cite{DT, Adams}. 

The optimal Orlicz target problem has been solved in general in \cite{Cianchi:Indiana} for $m=1$ (see also  \cite{Cianchi:Comm} for an alternate formulation of the solution), and in \cite{Cianchi:Forum} for arbitrary $m \in \mathbb N$.  As shown in these papers, given any Orlicz-Sobolev space $W^{m,A}_0(\Omega)$, there always exists an optimal target Orlicz space $L^B(\Omega)$ in \eqref{sob0}, and the function $B$ admits an explicit expression in terms of $A$, $n$ and $m$.
Thus, the  class of Orlicz spaces  is closed under the operation of associating an optimal target in Sobolev embeddings. By contrast, this property is not enjoyed by the smaller family of Lebesgue spaces, namely in the context of classical Sobolev embeddings. Actually, if  $A(t) = t^p$ for some $p \geq 1$, so that $W^{m,A}_0 (\Omega)$ agrees with the usual Sobolev space $W^{m,p}_0 (\Omega)$, and   $|\Omega|<\infty$,  one has that
\begin{equation}
	\label{classical}
	W^{m,p}_0 (\Omega) \to
	\begin{cases} L^{\frac{mp}{n-mp}}(\Omega) &
		\text{if $1 \leq m <n$ and $1 \leq p < \tfrac nm$,}
		\\
		\exp L^{\frac n{n-m}}(\Omega)  & \text{if $1 \leq m <n$ and $p =
		\tfrac nm$,}
		\\
		L^\infty (\Omega)  & \text{if either $1 \leq m <n$ and $ p > \tfrac
		nm$, or $ m \geq n$,}
	\end{cases}
\end{equation}
all targets being optimal in the class of Orlicz spaces. Here, $\exp L^{\frac n{n-m}}(\Omega)$ denotes the Orlicz space  associated with the Young function $e^{t^{\frac{n}{n-m}}}-1$. The first and the third embedding in \eqref{classical} are nothing but the classical Sobolev embedding. The second one was independently obtained by Yudovich \cite{Yudovich},  Pokhozhaev \cite{Pokhozhaev}, Strichartz \cite{Strichartz:Indiana}, and, for $m=1$,  by Trudinger \cite{Trudinger}. Note that, 
in the first and third embedding, the target is a Lebesgue space, and it is hence optimal also in this subclass, but no optimal Lebesgue target space exists in the second embedding.



The situation is different, and   subtler  in a sense, when the optimal Orlicz-Sobolev domain space $W^{m,A}_0 (\Omega)$ in \eqref{sob0}, for a given Young function $B$,  is in question. Actually, the existence of such an optimal domain is not guaranteed for every $B$. Testing the problem on the  spaces appearing in \eqref{classical} may help  have an idea of the possibilities that may occur. Assume that 
$$
	L^B(\Omega) = L^q(\Omega)
$$
for some $q \in [1, \infty]$. It is well known that, if $m \geq n$, then  $W^{m,1}_0(\Omega) \to L^\infty(\Omega)$, and hence, in particular,
\begin{equation}
	W^{m,1}_0(\Omega) \to L^q(\Omega),
	\label{classicaldom1}
\end{equation}
for every $q \in [1, \infty]$. Embedding \eqref{classicaldom1} continues to hold even if $1\leq m <n$, provided that  $q \leq \frac{n}{n-m}$. On the other hand,  if $1\leq m <n$ and $\frac{n}{n-m}<q<\infty$, then
\begin{equation}
	W^{m,\frac{nq}{n+mq}}_0(\Omega) \to L^q(\Omega).
	\label{classicaldom2}
\end{equation}
Both domain spaces in \eqref{classicaldom1} and \eqref{classicaldom2} are optimal among all Orlicz-Sobolev spaces \citep[Example 5.2]{Musil}. Instead, if $1 \leq m <n$,
\begin{equation}
	\text{``no optimal Orlicz-Sobolev space''} \to \exp L^{\frac n{n-m}}(\Omega)
	\label{classicaldom3}
\end{equation}
and
\begin{equation}
	\text{``no optimal Orlicz-Sobolev space''} \to  L^{\infty}(\Omega),
	\label{classicaldom4}
\end{equation}
see \citep[Theorem 4.3]{NAFSA98} and \citep[Theorem~6.4 (ii)]{CianchiPick:AM}, respectively, for the case when $m=1$, and \citep[Example 5.1~(b)]{Musil} for arbitrary $m\in \mathbb N$.
Equation \eqref{classicaldom3} means that any Orlicz-Sobolev space that is continuously embedded into the Orlicz space $\exp L^{\frac n{n-m}}(\Omega)$ can be replaced with a strictly larger Orlicz-Sobolev space which is still continuously embedded into $\exp L^{\frac n{n-m}}(\Omega)$. Equation \eqref{classicaldom4}, as well as similar statements about non-existence of optimal Orlicz-Sobolev domain spaces in what follows, has to be interpreted in an analogous sense.
In particular, interestingly enough, the space $W^{m, \frac nm}_0(\Omega)$, appearing on the left-hand side of \eqref{classical} when $p=\tfrac nm$, turns out  not to be optimal for Orlicz-Sobolev embeddings into $\exp L^{\frac n{n-m}}(\Omega)$.

As far as we know, these are the only instances for which the answer to   the optimal Orlicz-Sobolev domain problem is available in the literature.  The recent contribution \cite{Musil} provides a solution to an analogous problem for Orlicz-Sobolev embeddings of weak type, namely into  Marcinkiewicz  spaces.

Our aim here is to fill in this gap, and to address this question in full generality. We establish  a necessary and sufficient condition on the Young function $B$ for an optimal Orlicz-Sobolev domain $W^{m,A}_0(\Omega)$ to exist in \eqref{sob0}. Moreover, we exhibit   the optimal Young function $A$ when such an optimal domain   does exist. This is the content of Theorem~\ref{thm:W0}.

In fact, as mentioned above, our analysis is not confined to \eqref{sob0}, but also includes other related embedding problems. A natural variant amounts to
\begin{equation}
	W^{m, A}(\Omega) \to L^B(\Omega),
	\label{sob}
\end{equation}
where  $W^{m, A}(\Omega)$ is an Orlicz-Sobolev space of functions that are not subject to any  boundary condition. Under suitable regularity assumptions on $\Omega$, which are indispensable even in the classical Sobolev embedding, we show that the conclusions are exactly the same as for \eqref{sob0} -- see Theorem~\ref{thm:domainJohn}. 


Embeddings of the form \eqref{sob}, with $\Omega = \rn$, namely
\begin{equation}\label{sobrn}
	W^{m, A}(\rn) \to L^B(\rn),
\end{equation}
are the subject of Theorem~\ref{thm:domainRn}. The point here is that, unlike the case of sets $\Omega$ of finite measure, the behavior of the Young functions $A$ and $B$ near $0$ plays a role as well.

Finally, in Theorem~\ref{thm:Traces} the more general issue is faced of optimal Orlicz-Sobolev domains for embeddings into Orlicz spaces with respect to a Frostman measure $\mu$ on $\bar \Omega$. These read
\begin{equation}
	\label{sobtrace}
	W^{m, A}(\Omega) \to L^B(\bar\Omega , \mu),
\end{equation}
where $\Omega$ is a bounded Lipschitz domain, $\bar \Omega$ denotes the closure of $\Omega$, and $\mu$ is a Borel measure on $\bar \Omega$ such that 
\begin{equation}\label{E:sup}
\mu\bigl( B_r(x)\cap \bar{\Omega} \bigr) \leq C r^{\gamma} \quad 	\hbox{for every $x\in \rn$ and $r>0$,} 
\end{equation}
for some constants $C>0$ and $\gamma \in [n-m, n]$. Here, 
$B_r(x)$ denotes the ball centered at $x$, with radius~$r$. The restriction
$\gamma \geq n-m$ is imposed to guarantee that a trace operator on $\bar\Omega$,
endowed with the measure $\mu$, be well defined  on the space $W^{m, A}(\Omega)$, whatever  the Young function $A$ is. 

Of course, measures $\mu$ supported in $\Omega$, and hence embeddings into
Orlicz spaces $ L^B(\Omega , \mu)$, are included as special cases. On the other
hand, measures $\mu$ supported in $\partial \Omega$ correspond to trace
inequalities in a classical sense. In particular, on denoting by $\mathcal H^\gamma $ the $\gamma$-dimensional
Hausdorff measure, the choice  $\mu = \mathcal H^{n-1}|_{\partial \Omega}$
turns \eqref{sobtrace} into the boundary trace embedding 
\begin{equation}
	 W^{m, A}(\Omega) \to L^B(\partial \Omega)
	\label{sobboundary}
\end{equation}
enucleated in Corollary~\ref{cor:boundary}. Another customary  specialization
of $\mu$ amounts to the case when $\mu = \mathcal H^{d}|_{\Omega \cap {\mathcal
N}_d}$, where $d\in \mathbb N$, and ${\mathcal N}_d$ denotes a $d$-dimensional
compact submanifold of $\rn$. Embedding \eqref{sobtrace}  takes the form 
\begin{equation}\label{subspaces}
	 W^{m, A}(\Omega) \to L^B(\Omega \cap {\mathcal N}_d)
\end{equation}
in this case, with $d \in [n-m, n]$, see Corollary~\ref{cor:subspaces}.
Clearly, $\Omega \cap {\mathcal N}_d$ can, in particular, equal the
intersection of $\Omega$ with a $d$-dimensional affine subspace of $\rn$.

The results mentioned above are stated in Section~\ref{main}, that also contains applications to special instances of  Orlicz spaces. The necessary background material is collected in Section~\ref{back}. Section~\ref{indices} is devoted to certain properties and relations among the Boyd indices of the Young functions that play a role in our analysis.  Proofs of the main results are presented in the final Section~\ref{proofs main}.

\section{Background}\label{back}


\subsection{Young functions}

We call $A\colon [0, \infty) \to [0, \infty]$  a Young function if it is  convex, left-continuous, and $A(0)=0$. Any function of this kind satisfies, in particular,
\begin{equation}\label{kt}
kA(t) \leq A(kt)
\quad \text{if $k \geq 1$ and $t \geq 0$.}
\end{equation}
The Young conjugate 
 $\widetilde{A}$ of  $A$ is given by
\begin{equation*}
    \widetilde{A}(t)=\sup \{st-A(s): s \geq 0\} \quad \text{for $t\ge 0$.}
\end{equation*}
The  function $\widetilde{A}$ is a Young function as well, and its Young conjugate is again $A$. 
One has that
\begin{equation} \label{eq:YoungCompl}
    t\le A^{-1}(t)\,\widetilde{A}^{-1}(t)\le 2t\quad \text{for $t\ge 0$,}
\end{equation}
where $A^{-1}$ denotes the generalized right-continuous inverse of $A$. 
The function $B$, defined as  $B(t) = cA(bt)$, where $b,c$ are positive constants,  is also a Young function and
\begin{equation} \label{YoungCompl}
    \widetilde{B}(t) = c\widetilde{A}\bigl(\textstyle\frac{t}{bc}\bigr) \quad \text{for $t\ge 0$.}
\end{equation}
A Young function $A$ is said to satisfy the $\Delta_2$-condition near infinity  [resp.\ near zero] [resp.\ globally] if it  is finite-valued and there exist constants $c>0$ and $t_0 > 0$ such that
\begin{equation*}
    A(2t) \le c A(t)
        \quad \text{for   $t \geq t_0$ \,\, [$0\leq t \leq t_0$] \,\, [$t\geq 0$].}
\end{equation*}
A Young function $A$ is said to dominate another Young function $B$ near infinity [near zero] [globally] if there exist   constants $c>0$ and $t_0 > 0$ such that
\begin{equation*}
B(t) \leq A(ct) \quad \text{for $t \geq t_0$ \,\, [$0\leq t \leq t_0$] \,\, [$t \geq 0$].}
\end{equation*}
The functions $A$ and $B$ are called equivalent near infinity [near zero] [globally] if they dominate each other near infinity [near zero]  [globally]. 
\\ More generally, the terminology \lq\lq near infinity",  \lq\lq near zero",  \lq\lq globally" will be adopted to indicate that some property of a function of $t$ holds for $t \geq t_0$, for $0\leq t \leq t_0$ or for  $t \geq 0$, respectively.

\subsection{Boyd indices}
Given a Young function $A$, we define the function $h_A^\infty\colon (0, \infty) \to [0, \infty)$ as 
\begin{equation*}
	h_A^\infty (t) = \sup_{s>0} \frac{A^{-1}(st)}{A^{-1}(s)} 
    \quad \text{for $t>0$.}
\end{equation*}
The global lower and upper Boyd indices of $A$ are then defined as
\begin{equation} \label{BIdef}
    i_A^\infty = \sup_{1<t<\infty} \frac{\log t}{\log h_A^\infty(t)}
        \quad\text{and}\quad
    I_A^\infty	= \inf_{0<t<1} \frac{\log t}{\log h_A^\infty(t)}\,,
\end{equation}
respectively. One has that
\begin{equation} \label{BIprop1}
    1\le i_A^\infty \le I_A^\infty \le \infty.
\end{equation}
It can also be shown that
\begin{equation} \label{BIprop2}
    i_A^\infty	= \lim_{t\to\infty} \frac{\log t}{\log h_A^\infty(t)}
        \quad\text{and}\quad
    I_A^\infty	= \lim_{t\to 0^+} \frac{\log t}{\log h_A^\infty(t)}.
\end{equation}

The Boyd indices of $A$ admit an alternate expression, that does not call into play $A^{-1}$, provided that $A$ is finite-valued. 
Define $\widehat{h}_A^\infty\colon (0, \infty) \to [0, \infty)$ as 
\begin{equation*}
	\widehat{h}_A^\infty (t) = \sup_{s>0} \frac{A(st)}{A(s)} 
    \quad \text{for $t>0$.}
\end{equation*}
Then, 
\begin{equation}
	i_A^\infty = \sup_{0<t<1} \frac{\log\widehat{h}_A^\infty(t)}{\log t}
		\quad\text{and}\quad
	I_A^\infty = \inf_{1<t<\infty} \frac{\log\widehat{h}_A^\infty(t)}{\log t}.
	\label{pre:BoydIndex}
\end{equation}
Furthermore,
the supremum and infimum in \eqref{pre:BoydIndex} can be replaced with the
limits as $t\to 0^+$ and $t\to\infty$, respectively.
\\
The local lower and upper Boyd indices $i_A$ and $I_A$ of $A$
are defined as in \eqref{BIdef}, with $h_A^\infty$ replaced by the function $h_A\colon (0, \infty) \to [0, \infty]$ given by
\begin{equation*}
	h_A(t) = \limsup_{s\to\infty} \frac{A^{-1}(st)}{A^{-1}(s)} 
    \quad \text{for $t>0$.}
\end{equation*}
Properties parallel to \eqref{BIprop1} and \eqref{BIprop2} hold, with 
$i_A^\infty$ and $I_A^{\infty}$ replaced by $i_A$ and $I_A$.
Moreover, on 
defining $\widehat{h}_A\colon (0, \infty) \to [0, \infty)$ as 
\begin{equation*}
	\widehat{h}_A(t) = \limsup_{s\to\infty} \frac{A(st)}{A(s)} 
    \quad \text{for $t>0$,}
\end{equation*}
a version of equation   \eqref{pre:BoydIndex} holds for $i_A$ and
$I_A$, with proper replacements, namely
\begin{equation}
	i_A = \sup_{0<t<1} \frac{\log\widehat{h}_A(t)}{\log t}
		\quad\text{and}\quad
	I_A = \inf_{1<t<\infty} \frac{\log\widehat{h}_A(t)}{\log t}.
	\label{march1}
\end{equation}
 Observe that if the function $A^{-1}(t)\,t^{-\sigma}$ is equivalent globally [near infinity], up to multiplicative
positive constants,  to a non-decreasing function, for some $\sigma \in (0,1)$,
then $I_A^\infty \le 1/\sigma$ [$I_A\le 1/\sigma$]. Similarly, if the function $A^{-1}(t)\,t^{-\sigma}$ is equivalent globally [near infinity]
to a non-increasing function, then $i_A^\infty\ge 1/\sigma$
[$i_A \ge 1/\sigma$].
\\
In the special case when 
 $A(t)=t^p$ for some $p\geq 1$, one has that $i_A^\infty=I_A^\infty =p$; furthermore,   if $A(t)=
\infty$ for large $t$, then $i_A=I_A=\infty$.
\\
We refer the reader to \cite{Boyd:Pacific} for more details on the material of this subsection.

\subsection{Orlicz spaces}

Let $\mathcal R$ be a sigma-finite, non-atomic, measure space endowed with a measure $\nu$. 
Denote by $\mathcal{M}(\mathcal R)$ the space of real-valued $\nu$-measurable functions in $\mathcal R$, and by $\mathcal{M}_+(\mathcal R)$ the set of nonnegative functions in $\mathcal{M}(\mathcal R)$. Given a 
Young function $A$, the Orlicz space $L^A(\mathcal R)$ is the collection  of all functions $f\in\mathcal{M}(\mathcal R)$ such that
\begin{equation*}
    \int_{\mathcal R} A\left(\frac{|f(x)|}{\lambda}\right)\,\d \nu (x) <\infty
\end{equation*}
for some $\lambda >0$.
The Orlicz space $L^A(\mathcal R)$ is a Banach space endowed with the Luxemburg norm defined as 
\begin{equation*}
    \|f\|_{L^A(\mathcal{R})}=\inf\left\{\lambda>0:
        \int_{\mathcal{R}}A\left(\frac{|f(x)|}{\lambda}\right)\,\d \nu (x) \le 1\right\}
\end{equation*}
for $f \in \mathcal M(\mathcal R)$.
The choice  $A(t)=t^p$, with  $1\le p < \infty$, yields $L^A(\mathcal R) = L^p(\mathcal R)$, the customary Lebesgue space. When $A(t)=0$ for $t \in [0,1]$ and $A(t)=\infty$ for $t \in (1, \infty)$, one has that $L^A(\mathcal R)=L^\infty (\mathcal R)$.

Let $E$ be a non-negligible measurable subset of $\mathcal R$, and let $\chi _E$ denote its characteristic function. Then 
\begin{equation}\label{chiE}
        \|\chi_E\|_{L^A(\mathcal R)} = \frac{1}{A^{-1}\bigl({\textstyle \frac{1}{|E|}}\bigr)}.
\end{equation}
The fundamental function $\varphi_{A}$ of $L^A(\mathcal R)$ is defined as 
\begin{equation*}
    \varphi_{A}(s) = \frac{1}{A^{-1}\bigl(\frac{1}{s}\bigr)} \quad \hbox{for $0<s<\nu (\mathcal R)$,}
\end{equation*}
and $\varphi_{A}(0)=0$.
Owing to \eqref{chiE},
\begin{equation*}
    \varphi_{A} (s) = \| \chi_E \|_{L^A(\mathcal R)}
       \end{equation*}
       for every set $E\subset \mathcal R$ such that
        $\nu(E)=s$.
       A H\"older type inequality in Orlicz spaces asserts that
       \begin{equation}\label{holder}
\|g\|_{L^{\widetilde A}(\mathcal R)}     \leq \sup_{f \in L^A(\mathcal R)} \frac{\int _{\mathcal R}f(x)g(x)\, \d\nu (x)}{\|f\|_{L^{ A}(\mathcal R)} } \leq 2 \|g\|_{L^{\widetilde A}(\mathcal R)} 
       \end{equation}
       for every $g \in L^{\widetilde A}(\mathcal R)$.

The inclusion relations between Orlicz spaces can be characterized in terms of the notion of domination between Young functions. Assume that $\nu (\mathcal R) < \infty$ [$\nu (\mathcal R) = \infty$], and let $A$ and $B$ be Young functions. Then 
\begin{equation}\label{inclusion}
\hbox{$L^{A}(\mathcal R) \to L^{B}(\mathcal R)$ if and only if 
$A$ dominates $B$ near infinity [globally].}
\end{equation}

The alternate notation $A(L)(\mathcal R)$ for the Orlicz space $L^A(\mathcal R)$ will be adopted when convenient. In particular, 
if $\nu (\mathcal R) < \infty$, and $A(t)$ is equivalent to $t^p (\log (1 +t))^\alpha$ near infinity, where either $p>1$ and $\alpha \in \mathbb R$, or $p=1$ and $\alpha \geq 0$, then the Orlicz space $L^A(\mathcal R)$ is the so-called Zygmund space  denoted by $L^p(\log L)^\alpha (\mathcal R)$.  Orlicz spaces of exponential type are denoted by $\exp L^\beta (\mathcal R)$, and are built upon the Young function $A(t)= e^{t^\beta}-1$, with $\beta >0$. 

\subsection{Marcinkiewicz spaces}

We denote by $M^A(\mathcal R)$ the weak Orlicz space associated with $A$, namely the Marcinkiewicz type space 
endowed with the norm obeying
\begin{equation*}
    \|f\|_{M^A(\mathcal R)}
        = \sup_{0<s<\nu (\mathcal R) } \frac{f^{**}(s)}{A^{-1}(\textstyle \frac{1}{s})}
\end{equation*}
for $f\in\mathcal{M}(\mathcal R)$. Here, $f^{**}\colon (0, \infty) \to [0, \infty]$ is the  function defined as  
\begin{equation*}
    f^{**}(s) = \frac{1}{s} \int_{0}^{s} f^*(r)\,\d r \quad \hbox{for $s >0$,}
\end{equation*}
where $f^*\colon [0, \infty) \to [0, \infty]$ denotes the decreasing rearrangement of $f$ given by
\begin{equation*}
    f^*(s)=\inf
        \bigl\{t>0:
            \bigl|\bigl\{x\in\mathcal R: |f(x)|>t\bigr\}\bigr|\le s
        \bigr\}\quad \hbox{for $s \geq 0$.}
\end{equation*}
Since  $f^*(s) \leq f^{**}(s)$  for $s>0$,
\begin{equation}\label{f*f**}
\sup_{0<s<\nu (\mathcal R) } \frac{f^{*}(s)}{A^{-1}(\textstyle \frac{1}{s})} \leq \|f\|_{M^A(\mathcal R)}
\end{equation}
for every $f \in \mathcal R$. The associate space to $M^A(\mathcal R)$ is denoted by $(M^A)'(\mathcal R)$, and endowed with the norm given by
\begin{equation}\label{MB'}
\|f\|_{(M^A)'(\mathcal R)}= 
\sup_{g \in M^A(\mathcal R)} \frac{\int _{\mathcal R}f(x)g(x)\, \d\nu (x)}{\|g\|_{M^{ A}(\mathcal R)} }\,
\end{equation}
for $f \in \mathcal{M}(\mathcal R)$. 
\\
The Orlicz space built upon a Young function $A$ is embedded  into the corresponding Marcinkiewicz space, namely 
\begin{equation*}
    L^A(\mathcal R) \to M^A(\mathcal R)
\end{equation*}
for every Young function $A$. Moreover, 
\begin{equation}\label{phiM}
\|\chi_E\|_{M^A(\mathcal R)} \simeq \varphi _A(|E|)
\end{equation}
for every measurable set $E \subset \mathcal R$. Here, and in what follows, the relation $\simeq$ between two expressions means that they are bounded by each other, up to multiplicative positive constants independent of the involved relevant variables. 



\subsection{Orlicz-Sobolev spaces}

Let $n\in\N$, $n\geq 2$, and let $\Omega$ be an open subset of $\R^n$. Given $m\in\N$ and a Young function $A$, the $m$-th order Orlicz-Sobolev space built upon $A$  is defined by
\begin{multline*}
    W^{m,A}(\Omega) =
        \Bigl\{u\in\mathcal{M}(\Omega): u\text{ is $m$-times weakly differentiable in $\Omega$, and } \\
            |\nabla^k u|\in L^A(\Omega),\, k=0,1,\ldots, m\Bigr\}.
\end{multline*}
Here, $\nabla^k u$ denotes the vector of all $k$-th order weak derivatives of $u$ and $\nabla^0 u = u$.  One has that $W^{m,A}(\Omega)$ is a Banach space equipped with the norm defined as
$$
    \|u\|_{W^{m,A}(\Omega)} =
        \sum_{k=0}^{m} \,\|\nabla^k u \|_{L^A(\Omega)}
$$
for $u\in W^{m,A}(\Omega)$.
By $W_0^{m,A}(\Omega)$ we denote the subspace of $W^{m,A}(\Omega)$ of those functions $u$ in $\Omega$ whose continuation by $0$ outside $\Omega$ belongs to $W^{m,A}(\R^n)$.
The  notations $W^mL^A(\Omega)$ and  $W^m A(L)(\Omega)$ 
will also be occasionally adopted instead of $W^{m,A}(\Omega)$; analogous alternate notations will be used  for $W^{m,A}_0(\Omega)$.
\\
If $|\Omega|< \infty$, an iterated use of a Poincar\'e type inequality in Orlicz spaces \cite[Lemma 3]{Talenti90} ensures that the functional 
$$
	\|\nabla^m u\|_{L^A(\Omega)}
$$
defines a norm on $W_0^{m,A}(\Omega)$ equivalent to  $\|u\|_{W^{m,A}(\Omega)} $.

As in the case of Orlicz spaces, inclusion relations between Orlicz-Sobolev spaces can be described in terms of domination  between the defining Young functions $A$ and $B$. If $|\Omega|< \infty$, then  
\begin{equation}\label{inclusionsob}
W^{m,A}(\Omega) \to W^{m,B}(\Omega)
\quad \bigl[W^{m,A}_0(\Omega) \to W^{m,B}_0(\Omega)\bigr] \quad \hbox{if and only if 
$A$ dominates $B$ near infinity.}
\end{equation}
On the other hand,  
\begin{equation}\label{inclusionsobrn}
W^{m,A}(\rn) \to W^{m,B}(\rn) \quad \hbox{if and only if 
$A$ dominates $B$ globally.}
\end{equation}
A proof of assertions \eqref{inclusionsob} and \eqref{inclusionsobrn} seems not to be available in the literature. We sketch a proof  in Proposition~\ref{P:inclusions}, Section~\ref{proofs main}.

Sobolev and trace embeddings for functions with unrestricted boundary values require some regularity on the ground domain. The class of John domains is known to be essentially the largest where Sobolev type embeddings hold in their strongest form.  A bounded open set $\Omega\subset \rn$ is called a John domain if there exist a constant $c\in (0,1)$ and a point $x_0\in\Omega$ such that for every $x\in\Omega$ there exists a rectifiable curve $\varpi\colon [0,l]\to\Omega$, with $l >0$, parametrized by arclength, such that $\varpi(0)=x$, $\varpi(l)=x_0$, and
$$
    \dist\bigl(\varpi(r),\partial\Omega\bigr) \ge cr\quad  \hbox{for $r\in[0,l]$.}
$$
The class of John domains includes  classical families of open sets, such as that of bounded Lipschitz domains, and that of domains with the cone property.
Recall that a bounded open set $\Omega$ is said to have the cone property if there exists a finite circular cone $\Lambda$ such that each point in $\Omega$ is the vertex of a finite cone contained in $\Omega$ and congruent to $\Lambda$.

\subsection{Reduction principles}

A key ingredient in our approach is the use of so-called reduction principles for Sobolev type embeddings. They assert that a wide class of Sobolev and trace inequalities, including those considered in this paper, are in fact equivalent to considerably simpler one-dimensional inequalities for suitable Hardy type operators. The relevant operators are defined as
\begin{equation}\label{Hab}
    H_{\alpha, \beta} f (s)
        = \int_{s^\beta}^{1} f(r)\,r^{\alpha - 1}\,\d r
        \quad \text{for $s>0$}
\end{equation}
for any function $f\in \mathcal M(0,1)$ making the integral in \eqref{Hab} converge. The exponents $\alpha $ and $\beta$ satisfy the constraints $0< \alpha <1$, $0<\beta<\infty$ and $\alpha+1/\beta \ge 1$, and depend on the Sobolev inequality in question.

 Given any open set
$\Omega \subset \rn$ with $|\Omega |<\infty$, embedding \eqref{sob0} is equivalent to the inequality
\begin{equation}
	\|u\|_{L^B(\Omega)}\leq C_1  \|\nabla ^mu\|_{L^{A}(\Omega)}
	\label{sobineq0}
\end{equation}
for some constant $C_1$ and for every $u\in W^{m,A}_0(\Omega)$.
The pertinent  reduction principle asserts that inequality \eqref{sobineq0} holds if and only if 
\begin{equation}
	\|H_{\frac mn, 1}f\|_{L^B(0,1)}\leq C_2  \|f\|_{L^A(0,1)}
	\label{hardy0}
\end{equation}
for some constant $C_2$, and for every nonnegative $f\in L^A(0,1)$. See 
\citep[Proof of Theorem 1]{Cianchi:Indiana} for $m=1$, and \citep[Theorem~A]{KP}  and \citep[Theorem~6.1]{CPS} for arbitrary $m$. Moreover,
the constants $C_1$ and $C_2$ depend on each other, and on $n$, $m$ and
$|\Omega|$.

Embedding \eqref{sob} in a John domain $\Omega$ amounts to the inequality 
\begin{equation}
	\|u\|_{L^B(\Omega)}\leq C_1  \|u\|_{W^{m,A}(\Omega)}
	\label{sobineqjohn}
\end{equation}
for every $u\in W^{m,A}(\Omega)$. Inequality  \eqref{sobineqjohn} is again equivalent to \eqref{hardy0} (\citep[Proof of Theorem 2]{Cianchi:Indiana} for $m=1$, and  \cite[Theorem~6.1]{CPS} for any $m$). However, in this case the mutual dependence of  the constants $C_1$ and $C_2$ involves full information on $\Omega$, and not just on $|\Omega|$. 

A characterization of embeddings on the whole   $\rn$ requires a combination of the Hardy inequality \eqref{hardy0}, which only depends on the behavior of the functions $A$ and $B$ near infinity, with a condition on their decay near zero. Specifically, the inequality
\begin{equation}
	\|u\|_{L^B(\rn)}\leq C  \|u\|_{W^{m,A}(\rn)}
	\label{sobineqrn}
\end{equation}
holds for some constant $C$, and for every $u\in W^{m,A}(\rn)$ if and only if  inequality  \eqref{hardy0} holds, and 
\begin{equation}\label{dic21}
\hbox{$A$ dominates $B$ near zero, }
\end{equation}
see \cite{ACPS}.
\par
The reduction principle for  embedding \eqref{sobtrace} into Orlicz spaces,  with respect to Frostman measures, 
applies to bounded Lipschitz domains $\Omega$ in $\rn$. It provides us with a sufficient condition for the validity of 
\eqref{sobtrace} in terms of an appropriate Hardy type inequality, and it is also necessary if the decay in \eqref{E:sup} is sharp, in the sense that there exist $x_0\in \bar \Omega$ and positive constants $c$ and $R>0$ such that
\begin{equation}\label{E:inf}
\mu (B_r(x_0))\cap \bar \Omega) \geq c r^\gamma \quad \hbox{if $0<r<R$.}
\end{equation}
The relevant principle asserts that, if  \eqref{E:sup} and \eqref{E:inf} are in force for some $\gamma \in [n-m, n]$, then the inequality
\begin{equation}
	\|u\|_{L^B(\bar \Omega, \mu)}\leq C_1  \|u\|_{W^{m,A}(\Omega)}
	\label{traceineq}
\end{equation}
holds for some constant $C_1$ and for every $u\in W^{m,A}(\Omega)$ if and only if
\begin{equation}
	\|H_{\frac mn, \frac n\gamma}f\|_{L^B(0,1)}\leq C_2  \|f\|_{L^A(0,1)}
	\label{hardytrace}
\end{equation}
 for some constant $C_2$, and for every nonnegative $f\in L^A(0,1)$.  The constants $C_1$ and $C_2$ depend on each other, and on $n$, $m$, $\gamma$, $\Omega$ and on the constants appearing in \eqref{E:sup} and \eqref{E:inf}. The equivalence of  inequalities \eqref{traceineq} and  \eqref{hardytrace} is established in \cite{CPStrace}. Let us mention that the special case when $\mu$ is the $(n-1)$-dimensional Hausdorff measure on $\partial \Omega$ is treated in \cite{CKP}. The case when $\gamma \in \mathbb N$, and $\mu$ is the $\gamma$-dimensional Hausdorff measure restricted to a $\gamma$-dimensional affine subspace of $\rn$ is dealt with in \cite{CianchiPick:TAMS}.

\section{Main results}\label{main}

Let us begin by considering embedding \eqref{sob0}. As a  preliminary observation, 
note that, when
\begin{equation}\label{finite1}
 m \geq n,
 \end{equation}
the optimal Orlicz-Sobolev domain  $W^{m,A}_0(\Omega)$  in \eqref{sob0} corresponds to the choice
\begin{equation*}
A(t) = t \quad \text{for $t \geq 0$,}
\end{equation*}
namely
\begin{equation*}
W^{m,A}_0(\Omega) = W^{m,1}_0(\Omega).
\end{equation*}
Indeed, under assumption \eqref{finite1}, one classically has
\begin{equation*}
W^{m,1}_0(\Omega) \to L^\infty (\Omega),
\end{equation*}
whence the optimality of $W^{m,1}_0(\Omega)$ follows,  since
\begin{equation}\label{trivial3}
W^{m,A}_0(\Omega) \to  W^{m,1}_0(\Omega) \to L^\infty (\Omega) \to
L^B(\Omega)
\end{equation}
for any Young functions $A$ and $B$.
\par
We may thus restrict our attention to the case when $$1 \leq m <n.$$
In this circumstance, the existence of an optimal Orlicz-Sobolev
space $W^{m,A}(\Omega)$ in \eqref{sob0} is not guaranteed anymore.
Our first main result asserts that 
the existence of such an optimal space depends on the local upper Boyd index of 
the Young function
$B_n$ given by
\begin{equation}\label{A}
	B_n(t) = \int _0^t \frac {G_n^{-1}(s)}{s} \d s
		\quad \text{for $t \geq 0$,}
\end{equation}
where $G_n\colon [0, \infty ) \to [0, \infty)$ is defined as
\begin{equation*}
	 G_n(t) = t \inf_{1\leq s\le t} B^{-1}(s)\,s^{\frac{m}{n}-1}
	 \quad \text{for $t \geq 1$}.
\end{equation*}
and  $ G_n(t) = t B^{-1}(1)$ for $t \in [0, 1)$. Moreover, whenever it exists, the function $A$ in the optimal domain space  in  \eqref{sob0} equals $B_n$.

\begin{remark}\label{rem:AG}
Observe that the function $G_n$ is increasing, as shown via the alternate formula  
\begin{equation*}
	G_n(t) = t^{\frac{m}{n}} \inf_{1\le s<\infty}
		B^{-1}(s)\max\bigl\{ 1, \tfrac ts \bigr\}^{1-\frac{m}{n}}
		\quad\text{for $t\ge 1$,}
\end{equation*}
and hence its inverse $G_n^{-1}$ is well-defined.
\\
Also, the function $B_n$ is actually a Young function. Indeed,  since $G_n$ is
increasing,   $G_n^{-1}$ is
increasing as well. Thus, since the function $G_n(t)/t$ is non-increasing, the
function $G_n^{-1}(t)/t$ is non-decreasing. These facts also ensure that $B_n$
is equivalent to $G_n^{-1}$ globally.
\end{remark}

\begin{theorem}[Optimal Orlicz-Sobolev domain under vanishing boundary conditions] \label{thm:W0}
Let $n \geq 2$ and $1 \leq m <n$, and let $B$ be a Young function. Let $B_n$ be
the Young function defined by \eqref{A}.  Assume that $\Omega$ is an open set
in $\mathbb R^n$ with $|\Omega|< \infty$. If
\begin{equation}
\label{eq:Johnindex}
	I_{B_n} < \frac{n}{m},
\end{equation}
then
\begin{equation}\label{emb0bis}
	 W^{m, B_n}_0(\Omega) \to L^B(\Omega),
\end{equation}
and $W^{m,B_n}_0(\Omega)$ is the  optimal Orlicz-Sobolev domain space in \eqref{emb0bis}.
\\
Conversely, if \eqref{eq:Johnindex} fails, then no optimal Orlicz-Sobolev domain space  exists in \eqref{sob0}, in the sense that  any Orlicz-Sobolev space $W^{m,A}_0(\Omega)$ for which embedding \eqref{sob0} holds can be replaced with   a strictly larger Orlicz-Sobolev space for which \eqref{sob0} is still true.
\\
In particular, if $i_B  > \textstyle\frac{n}{n-m}$, then condition
\eqref{eq:Johnindex} is equivalent to $I_B < \infty$, and
\begin{equation}
	B_n^{-1}(t)
		\simeq B^{-1}(t)\,t^{\frac mn}
		\quad \text{near infinity.}
	\label{oct2}
\end{equation}
\end{theorem}

Under a mild additional assumption on the decay of $B$ near $0$, which reads
\begin{equation} \label{Bz}
	\inf_{0<t<1} \frac{B(t)}{t^\frac{n}{n-m}} >0\,,
\end{equation}
embedding \eqref{emb0bis} is equivalent to a Sobolev inequality in integral form. Inequalities in this form are usually better suited for applications to the theory of partial differential equations. The relevant integral inequality requires a slight variant in the definition near $0$ of the Young function in the optimal Orlicz-Sobolev domain. This function will be denoted by $ B_n^\infty$, and is defined as  
\begin{equation}\label{Aglob}
	 B_n^\infty(t) = \int _0^t \frac {{G_n^\infty}^{-1}(s)}{s} \d s
		\quad \text{for $t \geq 0$,}
\end{equation}
where
\begin{equation*}
	 G_n^\infty(t) = t \inf_{0<s\le t} B^{-1}(s)\,s^{\frac{m}{n}-1}
	 \quad \text{for $t \geq 0$}.
\end{equation*}
Note that condition \eqref{eq:Johnindex} on the local upper Boyd index of $B_n$ has now to be replaced with a parallel condition on the global upper Boyd index of $B_n^\infty$.

\begin{corollary} \label{cor:W0}
Let $n$, $m$ and $\Omega$ be as in Theorem \ref{thm:W0}.  Let $B$ be a Young function
satisfying \eqref{Bz}, and let $B_n^\infty$ be
the Young function defined by \eqref{Aglob}.   If
\begin{equation} \label{eq:Johnindexglobal}
	I_{B_n^\infty}^\infty < \frac{n}{m},
\end{equation}
then there exists a constant~$C$ such that
\begin{equation}
	\int_\Omega B\Bigg(
			\frac{|u(x)|} {C\bigl(\int _\Omega  B_n^\infty(|\nabla ^m u|)\,\d y\bigr)^{m/n}}
		\Bigg)\,\d x
		\leq \int_\Omega  B_n^\infty(|\nabla ^m u|)\,\d x
	\label{oct11}
\end{equation}
for every $u \in W^{m, B_n^\infty}_0(\Omega)$.
\\
In particular, if $i_B^\infty > \textstyle\frac{n}{n-m}$, then condition
\eqref{eq:Johnindexglobal} is equivalent to $I_B^\infty < \infty$, and
\begin{equation*}
	 {B_n^\infty}^{-1}(t)
		\simeq B^{-1}(t)\,t^{\frac mn}
		\quad \text{for $t\ge 0$.}
\end{equation*}
\end{corollary}

Companion results to Theorem \ref{thm:W0} and Corollary \ref{cor:W0} hold for embedding \eqref{sob} between spaces of functions with unrestricted boundary values, provided that $\Omega$ is a John domain. Like for embedding \eqref{sob0}, the only non-trivial case is when $1 \leq m <n$. Indeed,   if $m \geq n$,   the same chain as in \eqref{trivial3} holds with $W^{m,A}_0(\Omega)$ and $W^{m,1}_0(\Omega)$ replaced by $W^{m,A}(\Omega)$ and $W^{m,1}(\Omega)$, respectively, and hence $W^{m,1}(\Omega)$ is the optimal Orlicz-Sobolev domain space in \eqref{sob}.

\begin{theorem}[Optimal Orlicz-Sobolev domain without boundary conditions] \label{thm:domainJohn}
	Let $n \geq 2$ and $1 \leq m <n$, and let $B$ be a Young
	function. Assume that $\Omega$ is a John  domain in $\rn$.
	Let $B_n$ be the Young function defined by \eqref{A}.  If
	\eqref{eq:Johnindex} holds, then
	 \begin{equation}
	 	\label{johnemb}
		W^{m, B_n}(\Omega) \to L^B(\Omega),
	\end{equation}
	and $W^{m,B_n}(\Omega)$ is the optimal Orlicz-Sobolev domain space
	in \eqref{johnemb}.
    \\
Conversely, if \eqref{eq:Johnindex} fails, then no optimal Orlicz-Sobolev domain space  exists in \eqref{sob}, in the sense that  any Orlicz-Sobolev space $W^{m,A}(\Omega)$ for which embedding \eqref{sob} holds can be replaced with   a strictly larger Orlicz-Sobolev space for which \eqref{sob} is still true.
\\
In particular, if $i_B > \textstyle\frac{n}{n-m}$, then condition \eqref{eq:Johnindex} is equivalent to $I_B < \infty$, and
\begin{equation*}
    B_n^{-1}(t)
        \simeq B^{-1}(t)\,t^{\frac mn}
        \quad \text{near infinity}.
\end{equation*}
\end{theorem}

\begin{remark}
An integral inequality analogous to \eqref{oct11}, corresponding to embedding \eqref{johnemb}, holds under assumption \eqref{Bz}, and with $B_n$ replaced by $B_n^\infty$.
\end{remark}

\begin{example}
\label{ex:LlogL} Consider the case when $L^B(\Omega)$ is a Zygmund space of the form $L^q(\log L)^\alpha (\Omega)$, where either $q \in (1, \infty)$ and $\alpha \in \mathbb R$, or $q=1$ and $\alpha \geq 0$.
Assume that $1 \leq m <n$, the only nontrivial case in view of the  discussion above.
Computations show that

\begin{equation*}
	B_n(t) \,\,\text{is equivalent to}\,
    \begin{cases}
        t^{\frac{nq}{n+mq}}\, (\log t)^{\frac{n\alpha}{n+mq}}
        & \text{if $q>\frac{n}{n-m}$, $\alpha \in\R$,} \\
        t\, (\log t)^{\alpha(1-\frac{m}{n})}
        & \text{if $q=\frac{n}{n-m}$, $\alpha > 0$,} \\
        t & \text{otherwise,}
    \end{cases}
\end{equation*}
near infinity. Moreover,
\begin{equation*}
	I_{B_n} = 
    \begin{cases}
    	\frac{nq}{n+mq}
        & \text{if $q>\frac{n}{n-m}$, $\alpha \in\R$,} \\
        1
        & \text{otherwise,}
    \end{cases}
\end{equation*}
whence $I_{B_n} < n/m$. Therefore, by Theorem~\ref{thm:W0},
\begin{equation}\label{41a}
    \begin{rcases}
        \text{if   $q>\frac{n}{n-m}$, $\alpha \in\R$,} \quad \quad \quad & W^{m}_0L^{\frac{nq}{n+mq}}(\log L)^{\frac{n\alpha}{n+mq}}(\Omega) \\
         \text{if   $q=\frac{n}{n-m}$, $\alpha > 0$,}  \quad \quad \quad &  W^{m}_0L(\log L)^{\alpha(1-\frac{m}{n})}(\Omega)
         \\ 
         \text{otherwise,} \quad \quad \quad & W^{m,1}_0(\Omega)  
    \end{rcases} \to L^q (\log L)^\alpha (\Omega)
\end{equation}	
for any open set $\Omega$ with $|\Omega|<\infty$,  and the domain spaces are optimal  among all Orlicz-Sobolev spaces. By Theorem~\ref{thm:domainJohn}, the same embeddings continue to hold, with optimal domain spaces, for any  John domain $\Omega$, provided that $W^m_0$ is replaced  by $W^m$. 
\\
Let us point out that, by \cite{Cianchi:Forum} (see also \citep{Cianchi:Indiana} for $m=1$), the space $L^q (\log L)^\alpha (\Omega)$ is in turn the optimal Orlicz target space in \eqref{41a}.
Thus, the domain and target spaces    are mutually optimal in \eqref{41a}. 
\end{example}

\begin{example}
We deal here  with  the target  space  $L^q  \exp \sqrt{\log L}(\Omega )$, with  $q\in [1,\infty)$, namely the Orlicz space built upon a Young function $B(t)= t^q e^{\sqrt{\log t}}$ near infinity. 
Assume as above that $1 \leq m<n$. If  $q<\frac{n}{n-m}$, then $B(t)\,t^\frac{n}{m-n} = t^{q+\frac{n}{m-n}}\,e^{\sqrt{\log t}}$ near infinity,  a decreasing function. Thus,
$B^{-1}(s)\,s^{\frac{m}{n}-1}$ is increasing near infinity, and $B_n(t)$ is equivalent to $t$ near infinity.
\\ Suppose next that $q\ge \frac{n}{n-m}$. Then the function $B(t)\,t^\frac{n}{m-n}$ is increasing near infinity, so that
$B^{-1}(s)\,s^{\frac{m}{n}-1}$ is decreasing, and
\begin{equation*}
	B_n^{-1}(t) \simeq  B^{-1}(t)\,t^\frac{m}{n}
\end{equation*}
near infinity.  One can verify that
$$B^{-1}(s) \simeq s^{\frac 1q} e^{- q^{-\frac 32}\sqrt {\log s}}$$
near infinity. Hence, 
$$B_n(t) \quad \text{is equivalent to}\quad t^{\frac{nq}{n+mq}}e^{\left(\frac{n}{n+mq}\right)^{\frac 32}\sqrt{\log t}}
$$
near infinity. In particular, 
$I_{B_n} = \frac{nq}{n+mq} < \frac{n}{m}$. Altogether, 
by Theorem~\ref{thm:W0}, one has that
\begin{equation*}
\begin{rcases}
        \text{if   $q\ge \frac{n}{n-m}$,} \quad \quad \quad & W^{m}_0 L^{\frac{nq}{n+mq}}\exp\bigl(\bigl(\frac{n}{n+mq}\bigr)^{\frac 32}\sqrt{\log L} \,\bigr)  (\Omega) \\
          \text{otherwise,
          } \quad \quad \quad
		 & W^{m,1}_0(\Omega) 
    \end{rcases} \to L^q  \exp \sqrt{\log L}(\Omega )
\end{equation*}
for any  open set $\Omega$ with $|\Omega|<\infty$, and the domain spaces are optimal among all Orlicz-Sobolev spaces.  A parallel result holds in any John domain $\Omega$, with $W^{m}_0$ replaced by $W^m$, owing to Theorem~\ref{thm:domainJohn}.
\end{example}

\begin{example}
If the Young function $B$ grows so fast near infinity that $i_B=\infty$, then it immediately follows from Theorems~\ref{thm:W0} and \ref{thm:domainJohn} that no optimal Orlicz-Sobolev domain space exists in embeddings \eqref{sob0} and \eqref{sob}. This is the case, for instance, when $L^B(\Omega)$ agrees with one of the following  spaces:
\begin{equation*}
   \exp \bigl((\log L)^\alpha\bigr)(\Omega)  \quad \exp \bigl(L^q (\log L)^\beta\bigr) (\Omega),
     \end{equation*}
     or
\begin{equation*}
	\exp L^\beta(\Omega),\,\,
	\exp\bigl(\exp L^\beta\bigr)(\Omega),\,
	\ldots,  \,  \exp\bigl( \cdots ( \exp L^\beta)\bigr)(\Omega),
	\end{equation*}
    or 
     \begin{equation*}
     L^\infty(\Omega)\,,
     \end{equation*}
where $\alpha>1$, $\beta>0$ and $q\in[1,\infty)$.
\end{example}

The next result is a counterpart of Theorem~\ref{thm:domainJohn}
in the case when $\Omega = \mathbb R^n$. The decay near zero of the involved Young functions is also
relevant now. A Young function $\bar B$ obeying
\begin{equation}\label{overB}
			\bar B(t) =
			\begin{cases}
				t & \text{near infinity,}\cr
				B(t) & \text{near zero,}\cr
			\end{cases}
		\end{equation}
and a Young function $\bar B_n$ obeying
\begin{equation}\label{overBD}
			\bar B_n(t) =
			\begin{cases}
				B_n(t) & \text{near infinity,}\cr
				B(t) & \text{near zero}\cr
			\end{cases}
		\end{equation}
come into play in the present situation.
\\ Let us stress that,  if $m \geq n$,
then the answer to the optimal domain problem is still easier than
in the case when $1 \leq m <n$, but not as trivial as when $\Omega$ is a John
domain, since the  optimal domain space is not
just $W^{m,1}(\rn)$ in general.

\begin{theorem} [Optimal Orlicz-Sobolev domain on $\rn$] \label{thm:domainRn}
Let $n \geq 2$ and $m \in \mathbb N$, and let $B$ be a Young function.
\\  {\rm (i)} Assume that $m \geq n$. Let $\bar B$ be a Young function satisfying \eqref{overB}.  Then
\begin{equation}\label{easyrn}
    W^{m,\bar B} (\rn) \to L^B(\rn)\,,
\end{equation}
and $W^{m,\bar B} (\rn)$ is the  optimal Orlicz-Sobolev domain space in \eqref{easyrn}. 
\\  {\rm (ii)}   Assume that $1 \leq m <n$.  Let $B_n$ be the Young function defined by \eqref{A}, and  let $\bar B_n$ be a Young function satisfying \eqref{overBD}. If  \eqref{eq:Johnindex} holds, then
\begin{equation} \label{embrn}
    W^{m,\bar B_n} (\rn) \to L^B(\rn)\,,
\end{equation}
and $W^{m,\bar B_n} (\rn)$ is the optimal Orlicz-Sobolev domain space in \eqref{embrn}.
\\
Conversely, if \eqref{eq:Johnindex} fails, then no optimal Orlicz-Sobolev domain space  exists in \eqref{sobrn},  in the sense that  any Orlicz-Sobolev space $W^{m,A}(\rn)$ for which embedding \eqref{sobrn} holds can be replaced with   a strictly larger Orlicz-Sobolev space for which \eqref{sobrn} is still true.
\\
In particular, if $i_B > \textstyle\frac{n}{n-m}$, then condition
\eqref{eq:Johnindex} is equivalent to $I_B < \infty$, and
\begin{equation*}
    \bar B_n^{-1}(t) \simeq
    \begin{cases}
        B^{-1}(t)\, t^{\frac mn} & \text{near infinity,}\cr
        B^{-1}(t) & \text{near zero.}\cr
    \end{cases}
\end{equation*}
\end{theorem}

Our last main results concern  the Orlicz-Sobolev  embedding \eqref{sobtrace} with a measure $\mu$ satisfying \eqref{E:sup} and \eqref{E:inf}. By the same reason as for \eqref{sob}, the optimal Orlicz-Sobolev domain space in these embeddings is $W^{m,1}(\Omega)$, provided that~$m \geq n$.

If, instead, $1\leq m <n$, the optimal Orlicz-Sobolev domain in \eqref{sobtrace}, when it exists, is built upon the Young function $B_{\gamma}$ defined, for  $\gamma \in [n-m, n]$, as
\begin{equation}
	\label{Atrace}
	B_{\gamma}(t) = \int _0^t \frac {G_{\gamma}^{-1}(s)}{s} \d s
    	\quad \text{for $t \geq 0$,}
\end{equation}
where $G_{\gamma}\colon [0, \infty ) \to [0, \infty)$ is given by
\begin{equation*}
	G_{\gamma} (t) = t \inf_{1\leq s\leq t} B^{-1}\bigl( s^{\frac \gamma n} \bigr) \,s^{\frac mn - 1}
		\quad \text{for $t \geq 1$,}
\end{equation*}
and $ G_{\gamma}(t) = t B^{-1}(1)$ for $t \in [0, 1)$. In particular, if $\mu$ is Lebesgue measure, then conditions \eqref{E:sup} and \eqref{E:inf} hold with $\gamma =n$, and $B_\gamma$ agrees with the function $B_n$ given by \eqref{A}.

\begin{theorem} [Optimal Orlicz-Sobolev domain  for  embeddings with measure] \label{thm:Traces}
Let  $n\ge 2$, and let $1 \leq m < n$.  Assume that $\Omega$ is a bounded Lipschitz domain in $\rn$, and let $\mu$ be a Borel measure satisfying conditions \eqref{E:sup} and \eqref{E:inf} for some $\gamma \in [n-m, n]$.  Let $B$ be a Young function, and let $B_{\gamma}$ be the Young function defined by \eqref{Atrace}. If
\begin{equation}
    \label{eq:Traceindex}
    I_{B_{\gamma}}< \frac nm,
\end{equation}
then
\begin{equation}\label{traceemb}
      W^{m, B_{\gamma}}(\Omega) \to L^B(\bar\Omega , \mu)\,,
\end{equation}
and $W^{m,B_{\gamma}}(\Omega)$ is the optimal Orlicz-Sobolev domain space in \eqref{traceemb}.
\\
Conversely, if \eqref{eq:Traceindex} fails, then no optimal Orlicz-Sobolev domain space  exists in \eqref{sobtrace}, in the sense that  any Orlicz-Sobolev space $W^{m,A}(\Omega)$ for which embedding \eqref{sobtrace} holds can be replaced with   a strictly larger Orlicz-Sobolev space for which \eqref{sobtrace} is still true.
\\
In particular, if $i_B> \textstyle\frac{\gamma}{n-m}$, then condition \eqref{eq:Traceindex} is equivalent to $I_B < \infty$, and
\begin{equation*}
    B_{\gamma}^{-1}(t)
        \simeq B^{-1}\bigl(t^{\frac \gamma n}\bigr)\,t^{\frac mn}
            \quad\text{near infinity.}
\end{equation*}
\end{theorem}

An integral version of embedding \eqref{traceemb} holds under the assumption that 
\begin{equation} \label{Bzt}
	\inf _{0<t<1} \frac {B(t)}{ t^\frac{\gamma}{n-m}} >0\,.
\end{equation}
It involves a modified version of the function $B_\gamma$ given by 
\begin{equation} \label{Atraceglob}
	B_{\gamma}^\infty (t) = \int _0^t \frac {{ G_{\gamma}^\infty}^{-1}(s)}{s} \d s
    	\quad \text{for $t \geq 0$,}
\end{equation}
where $G_{\gamma}^\infty\colon [0, \infty ) \to [0, \infty)$ is defined by
\begin{equation*}
	G_{\gamma}^\infty (t) = t \inf_{0<s\leq t} B^{-1}\bigl( s^{\frac \gamma n} \bigr) \,s^{\frac mn - 1}
		\quad \text{for $t \geq 0$.}
\end{equation*}

\begin{corollary} \label{cor:Traces}
Let $n$, $m$, $\gamma$, $\Omega$ and $\mu$ be as in Theorem \ref{thm:Traces}.
 Let $B$ be a Young function
satisfying \eqref{Bzt}, and let
$B_{\gamma}^\infty$ be the Young function defined by \eqref{Atraceglob}.
If
\begin{equation}
    \label{eq:Traceindexglobal}
    I_{B_{\gamma}^\infty}^\infty < \frac nm,
\end{equation}
then there exists a constant~$C$ such that
\begin{equation}
    \int_{\bar \Omega} B\Bigg(
        \frac{|u(x)|}
            {C\bigl(\sum _{k=0}^m\int _\Omega  B_{\gamma}^\infty(|\nabla ^k u|)\,\d y\bigr)^{m/n}}
        \Bigg)\,\d\mu (x) \leq 
    \bigg(\sum_{k=0}^m\int_\Omega  B_{\gamma}^\infty(|\nabla ^k u|)\,\d x\bigg)^{\frac \gamma n}
    \label{oct16}
\end{equation}
for every $u \in W^{m, B_{\gamma}^\infty}(\Omega)$.
\\
In particular, if $i_B^\infty > \textstyle\frac{\gamma}{n-m}$, then condition \eqref{eq:Traceindexglobal}
is equivalent to $I_B^\infty < \infty$, and
\begin{equation*}
     {B_{\gamma}^\infty}^{-1}(t)
        \simeq B^{-1}\bigl(t^{\frac \gamma n}\bigr)\,t^{\frac mn}
            \quad\text{for $t\ge 0$.}
\end{equation*}
\end{corollary}

\begin{example}
Assume  that $L^B(\Omega)= L^q (\log L)^\alpha(\Omega)$, the same Zygmund space as in Example~\ref{ex:LlogL}, where either $q\in(1,\infty)$ and $\alpha \in\R$, or $q=1$ and $\alpha \geq 0$. Assume that $1 \leq m< n$, the case when $m \geq n$ being trivial.  Let $\Omega$ be a bounded Lipschitz domain in $\mathbb R^n$, and let  $\mu$ be a Borel measure fulfilling conditions \eqref{E:sup} and \eqref{E:inf}.
Then
\begin{equation*}
	B_{\gamma}(t) \,\,\text{is equivalent to}\,
    \begin{cases}
        t^{\frac{nq}{\gamma+mq}}\, (\log t)^{\frac{n\alpha}{\gamma+mq}}
        & \text{if  $q>\frac{\gamma}{n-m}$, $\alpha \in\R$,} \\
        t\, (\log t)^{\frac{\alpha (n-m)}{\gamma}}
        & \text{if $q=\frac{\gamma}{n-m}$, $\alpha > 0$,} \\
        t
        & \text{otherwise,}
    \end{cases}
\end{equation*}
near infinity.
Hence,
\begin{equation*}
	I_{B_{\gamma}} = 
    \begin{cases}
    	\frac{nq}{\gamma+mq}
        & \text{if  $q>\frac{\gamma}{n-m}$, $\alpha \in\R$,} \\
        1
        & \text{otherwise.}
    \end{cases}
\end{equation*}
Since $I_{B_{\gamma}} < n/m$,  Theorem~\ref{thm:Traces} tells us that
\begin{equation*}
\begin{rcases}
        \text{if   $q>\frac{\gamma}{n-m}$, $\alpha \in\R$,}  \quad \quad \quad &   W^{m}   L^{\frac{nq}{\gamma+mq}}\, (\log L)^{\frac{n\alpha}{\gamma+mq}}   (\Omega) \\
        \text{if   $q=\frac{\gamma}{n-m}$, $\alpha > 0$,} &   W^{m} L (\log L)^{\frac{\alpha (n-m)}{\gamma}}(\Omega)
        \\
          \text{otherwise,} 
		 &   W^{m,1}(\Omega) 
    \end{rcases} \to L^q (\log L)^\alpha(\bar \Omega , \mu)\,,
\end{equation*}	
the domain spaces being optimal among all Orlicz-Sobolev spaces.
\end{example}

The optimal Orlicz-Sobolev domain space in \eqref{sobboundary} agrees with that in \eqref{sobtrace}, with $\gamma =n-1$. Namely, it is
built upon the Young function $B_{n-1}$ defined as in \eqref{Atrace}, with $\gamma = n-1$. This is the content of Corollary~\ref{cor:boundary} below, and follows from Theorem~\ref{thm:Traces}, and from the fact that, 
if $\Omega$ is a bounded Lipschitz domain, then the measure $\mu = \mathcal H^{n-1}|_{\partial \Omega}$ fulfills conditions  \eqref{E:sup} and \eqref{E:inf} with $\gamma = n-1$.

\begin{corollary}[Optimal Orlicz-Sobolev domain for boundary traces]\label{cor:boundary}
Let  $n\ge 2$ and $1 \leq m < n$. Assume that $\Omega$ is a bounded Lipschitz domain in $\rn$. Let $B$ be a Young function, and let $B_{n-1}$ be the Young function defined by \eqref{Atrace}, with $\gamma = n-1$.  If
\begin{equation}
    \label{eq:boundaryindex}
    I_{B_{n-1}} < \frac nm,
\end{equation}
then
\begin{equation}
    \label{boundaryemb}
      W^{m, B_{n-1}}(\Omega) \to L^B(\partial \Omega),
\end{equation}
and $W^{m,B_{n-1}}(\Omega)$ is the optimal Orlicz-Sobolev domain space in \eqref{boundaryemb}.
\\
Conversely, if \eqref{eq:boundaryindex} fails, then no optimal Orlicz-Sobolev domain space  exists in \eqref{sobboundary}, in the sense that  any Orlicz-Sobolev space $W^{m,A}(\Omega)$ for which embedding \eqref{sobboundary} holds can be replaced with   a strictly larger Orlicz-Sobolev space for which \eqref{sobboundary} is still true.
\\
In particular, if $i_B > \textstyle\frac{n-1}{n-m}$, then condition
\eqref{eq:boundaryindex} is equivalent to $I_B < \infty$, and
\begin{equation*}
    B_{n-1}^{-1}(t)
        \simeq B^{-1}\bigl(t^{\frac {n-1}n}\bigr)\,t^{\frac mn}
        \quad \text{near infinity.}
\end{equation*}
\end{corollary}

We conclude this section by specializing Theorem~\ref{thm:Traces} to embeddings of the form \eqref{subspaces} into Orlicz spaces defined on the intersection of $\Omega$ with $d$-dimensional compact submanifolds  ${\mathcal N} _d$ of $\rn$. Since the measure $\mu = \mathcal H^d|_{\Omega \cap {\mathcal N}_d}$ satisfies conditions \eqref{E:sup} and \eqref{E:inf}, with $\gamma = d$, from Theorem~\ref{thm:Traces} we infer the following corollary.

\begin{corollary}[Optimal Orlicz-Sobolev domain for traces on submanifolds]\label{cor:subspaces}
Let  $n\ge 2$, $1 \leq m < n$, and let $d \in \mathbb N$, with $n-m \leq d \leq n$. Assume that $\Omega$ is a bounded  Lipschitz domain in $\rn$, and let ${\mathcal N}_d$ be a $d$-dimensional compact submanifold  of $\rn$ such that $\Omega \cap {\mathcal N}_d \neq \emptyset$.
Let $B$ be a Young function, and let $B_{d}$ be the Young function defined by \eqref{Atrace}, with $\gamma = d$.  If
\begin{equation}
    \label{eq:subspaceindex}
    I_{B_{d}} < \frac nm,
\end{equation}
then
\begin{equation}
    \label{subspacesemb}
      W^{m, B_{d}}(\Omega) \to L^B(\Omega \cap {\mathcal N}_d)\,,
\end{equation}
and $W^{m,B_{d}}(\Omega)$ is the optimal Orlicz-Sobolev domain space in \eqref{subspacesemb}.
\\
Conversely, if \eqref{eq:subspaceindex} fails, then no optimal Orlicz-Sobolev domain space  exists in \eqref{subspaces}, in the sense that  any Orlicz-Sobolev space $W^{m,A}(\Omega)$ for which embedding \eqref{subspaces} holds can be replaced with   a strictly larger Orlicz-Sobolev space for which \eqref{subspaces} is still true.
\\
In particular, if $i_B > \textstyle\frac{d}{n-m}$, then condition
\eqref{eq:subspaceindex} is equivalent to $I_B < \infty$, and
\begin{equation*}
    B_{d}^{-1}(t)
        \simeq B^{-1}\bigl(t^{\frac {d}n}\bigr)\,t^{\frac mn}
        \quad \text{near infinity.}
\end{equation*}
\end{corollary}

\section{Boyd indices and optimal Orlicz domains } \label{indices}

This section is devoted to the analysis of certain properties of Young functions in connection with their Boyd indices. We begin with the following proposition, that collects various  characterizations  of pointwise and integral growth conditions   of a  Young function, and of its conjugate, in terms of their Boyd indices.

\begin{proposition} 
Let $E$ be a finite-valued Young function, and let $0<\alpha<1$. The following conditions are equivalent.
    \begin{enumerate}[\rm(i)]
    \item There exists a constant $k>1$ such that
    \begin{equation*}
        \int_t^\infty \frac{ E(s)}{s^{{1/\alpha}+1}}\,\d s
            \le \frac{ E(kt)}{t^{1/\alpha}}
                \quad\text{globally \quad [near infinity]}.
    \end{equation*}
    \item There exists a constant $k>1$ such that
    \begin{equation*}
        \int_0^t \frac{\widetilde{E}(s)}{s^{{1/(1-\alpha)}+1}}\,\d s
            \le \frac{\widetilde{E}(kt)}{t^{1/(1-\alpha)}}
                \quad\text{globally} \qquad \biggl[\int_1^t \frac{\widetilde{E}(s)}{s^{{1/(1-\alpha)}+1}}\,\d s
            \le \frac{\widetilde{E}(kt)}{t^{1/(1-\alpha)}}
                \quad\text{near infinity} \biggr].
    \end{equation*}
    \item There exist constants $\sigma>1$ and $c \in (0,1)$ such that
    \begin{equation*}
        E(\sigma t) \le c \sigma^{\frac{1}{\alpha}} E(t)
						\quad\text{globally \quad [near infinity]}.
    \end{equation*}
    \item There exist constants $\sigma>1$ and $c>1$ such that
    \begin{equation*}
        \widetilde{E}(\sigma t) \ge c\sigma^{\frac{1}{1-\alpha}} \widetilde{E}(t)
            \quad\text{globally \quad [near infinity]}.
    \end{equation*}
    \item The global [local] upper Boyd index of $E$ satisfies
    \begin{equation*}
        I_E ^\infty< 1/\alpha
				\qquad
        \bigl[I_E < 1/\alpha\bigr].
    \end{equation*}
    \item The global [local] lower Boyd index of $\widetilde{E}$ satisfies\textit{}
    \begin{equation*}
        i_{\widetilde E}^\infty > 1/(1-\alpha)
				\qquad
        \bigl[i_{\widetilde E} > 1/(1-\alpha)\bigr].
    \end{equation*}
    \end{enumerate}
\label{prop:Bconditions}
\end{proposition}

\begin{proof}
We shall prove the statement in the form \lq\lq near infinity". The proof of the global version is analogous - even simpler in fact - and will be omitted.
 \\
 \textit{(i) is equivalent to (iii)} This equivalence is stated in \citep[Lemma 2.3.~(ii)]{Stromberg:Indiana},  without proof.  We  provide a proof here, for
    completeness.
		Assume that there exist $k>1$ and $t_0> 0$ such that inequality (i) is fulfilled for  every $t>t_0$.
    Fix $\sigma>1$ and $t>t_0 k$, and let $\rho   \in [1, \sigma]$ be such that
    \begin{equation}
        E(\rho t)\, (\rho t)^{-\frac{1}{\alpha}}
            = \inf_{t\le r\le\sigma t}
            E(r)\, r^{-\frac{1}{\alpha}}.
            \label{defrho}
    \end{equation}
    We claim that
    \begin{equation}
        \sigma \ge \rho \ge e^{-k^{\frac{1}{\alpha}}} \sigma.
        \label{sigmarho}
    \end{equation}
    The former inequality is part of the definition of $\rho$. As for the latter, we have that
    \begin{equation*}
        E(\rho t)\, (\rho t)^{-\frac{1}{\alpha}} k^\frac{1}{\alpha}
            \ge \int_{\rho t/ k}^{\infty} \frac{E(s)}{s^{1/\alpha+1}} \d s
            \ge \int_{\rho t}^{\sigma t} \frac{E(s)}{s^{1/\alpha+1}} \d s
            \ge E(\rho t)\, (\rho t)^{-\frac{1}{\alpha}}  \log \Bigl( \frac{\sigma}{\rho} \Bigr),
    \end{equation*}
    whence the claim follows.
\\  Next, we show that $E$ satisfies the $\Delta_2$-condition near infinity. Suppose, by contradiction, that for every $j\in\N$ there exists  $t>t_0 k$ such that $E(2t) > j E(t)$. Choosing $\sigma = 2e^{k^{\frac{1}{\alpha}}}$, and $\rho$ defined by \eqref{defrho}, ensures that 
    \begin{equation}
        E(t) t^{-\frac{1}{\alpha}} k^\frac{1}{\alpha}
            \ge \int_{t/ k}^{\infty} \frac{E(s)}{s^{1/\alpha+1}} \d s
            \ge \int_{t}^{\sigma t} \frac{E(s)}{s^{1/\alpha+1}} \d s
            \ge E(\rho t)\, (\rho t)^{-\frac{1}{\alpha}}  \log\sigma\,.
            \label{lblog}
    \end{equation}
    Hence,
    \begin{equation*}
        E(2t)\, k^\frac{1}{\alpha}
            \ge j E(\rho t)\, \rho^{-\frac{1}{\alpha}} \log\sigma
            \ge j E(2t)\, \sigma^{-\frac{1}{\alpha}} \log\sigma,
    \end{equation*}
    since $\sigma\ge\rho\ge 2$, by \eqref{sigmarho}.  Therefore,
    \begin{equation*}
        k^\frac{1}{\alpha} \ge j\sigma^{-\frac{1}{\alpha}} \log\sigma
    \end{equation*}
    for all $j\in\N$, which is impossible.
\\ Now suppose that (iii) does not hold. Thus, for every $\sigma>1$ and   $c \in (0,1)$ there exists a sequence $\{t_j\}$ such that  $t_j \to  \infty$,  and
    \begin{equation}\label{dic1}
        E(\sigma t_j) > c \sigma^\frac{1}{\alpha} E(t_j)
    \end{equation}
    for $j \in \mathbb N$. Let $\rho$ be as in \eqref{defrho}.  By \eqref{dic1} and  \eqref{lblog},
    \begin{equation}\label{dic2}
        E(\sigma t_j)\, (\sigma t_j)^{-\frac{1}{\alpha}} k^\frac{1}{\alpha}
            > c E(t_j)\, t_j^{-\frac{1}{\alpha}} k^\frac{1}{\alpha}
            > c E(\rho t_j)\, (\rho t_j)^{-\frac{1}{\alpha}} \log\sigma.
    \end{equation}
    From \eqref{dic2}, \eqref{sigmarho} and the $\Delta_2$-condition near infinity for $E$, we conclude
    that there exists a positive constant $c_1$ such that
    \begin{equation*}
        E(\sigma t_j)\, k^\frac{1}{\alpha}
            > c E\bigl( e^{-k^\frac{1}{\alpha}} \sigma t_j\bigr) \log\sigma
            > c_1 E( \sigma t_j) \log\sigma
    \end{equation*} 
    for sufficiently large $j$. 
   Hence,
    \begin{equation*}
        k^\frac{1}{\alpha} > c_1 \log\sigma
    \end{equation*}
    for arbitrarily large $\sigma$,  a contradiction.
\\
		\textit{(iii) implies (i)}. Let $t_0>0$ be such that inequality (iii) holds for $t \geq t_0$. Let $j\in \N$. An iterative use of assumption (iii)  ensures that 
		\begin{equation}
			E(s) \le c^j \sigma^\frac{j}{\alpha} E(s\sigma^{-j}) 
				\quad \hbox{for $s\ge \sigma^j t_0$.}
				\label{eq:iiimtimes}
		\end{equation}
		 By \eqref{eq:iiimtimes}, if   $t\ge t_0$, then
		\begin{align*}
			\int_t^\infty \frac{ E(s)}{s^{{1/\alpha}+1}}\,\d s
				& = \sum_{j=0}^\infty
					\int_{t\sigma^j}^{t\sigma^{j+1}} \frac{ E(s)}{s^{{1/\alpha}+1}}\,\d s
				 \le \sum_{j=0}^\infty c^j
					\int_{t\sigma^j}^{t\sigma^{j+1}}
                    	\sigma^{j/\alpha} \frac{E(s\sigma^{-j})}{s^{1/\alpha+1}}\d s
					 \\ \noindent & = \sum_{j=0}^\infty c^j
					\int_{t}^{\sigma t} \frac{ E(r)}{r^{{1/\alpha}+1}}\,\d r
					\le \frac{1}{1-c}
					E(\sigma t)
					\int_{t}^{\sigma t} \frac{\d r}{r^{{1/\alpha}+1}}
                = \alpha \frac{1-\sigma^{1/\alpha}}{1-c}
					 \frac{ E(\sigma t)}{t^{1/\alpha}}\,.
		\end{align*}
		Hence, (i) follows via property \eqref{kt}.
\\	 \textit{(iii) implies (v)}. Assume that (v) does not hold, i.e.\ $I_E\ge 1/\alpha$.
		By equation \eqref{march1},  
		\begin{equation*}
			\frac{1}{\alpha} \le
				\inf_{1 <\sigma<\infty} \frac{\log \hat{h}_E(\sigma)}{\log\sigma},
		\end{equation*}
		and hence $\sigma^{1/\alpha} \le \hat{h}_E(\sigma)$ for every $\sigma\geq 1$.
		Owing to  the very definition of   $\hat{h}_E$,  
		\begin{equation*}
			\sigma^{\frac{1}{\alpha}}
				\le \limsup_{t\to\infty} \frac{E(\sigma t)}{E(t)}.
		\end{equation*}
		Hence, for every $c\in (0,1)$ and   $t_0>0$, there exists   $t>t_0$ such that
		\begin{equation*}
			c\sigma^{\frac{1}{\alpha}} < \frac{E(\sigma t)}{E(t)},
		\end{equation*}
		and this contradicts (iii).
\\  \textit{(v) implies (iii)}. Assume, by contradiction, that (iii) fails. Thereby, for every $\sigma>1$ and  $c \in (0,1)$ there exists a sequence
		$t_j \to \infty$ satisfying
		\begin{equation*}
			c\sigma^{\frac{1}{\alpha}} E(t_j) < E(\sigma t_j)
            	\quad \text{for $j \in \mathbb N$.}
		\end{equation*}
		Thus
		\begin{equation*}
			c\sigma^{\frac{1}{\alpha}}
				\leq  \limsup_{j\to\infty} \frac{E(\sigma t_j)}{E(t_j)}
				\le \limsup_{t\to\infty} \frac{E(\sigma t)}{E(t)}
				= \hat{h}_E (\sigma)\,,
		\end{equation*}
		whence
		\begin{equation*}
			\frac{\log \bigl( c\sigma^{1/\alpha} \bigr)}{\log \sigma}
				\leq \frac{\log \hat{h}_E(\sigma)}{\log \sigma}\,.
		\end{equation*}
		Thanks to \eqref{march1}, passing to the limit as $\sigma\to\infty$ yields $1/\alpha \le I_E$,
		thus contradicting  (v).
\\ \textit{(iii) is equivalent to (iv)}. 
Condition (iii) is equivalent to
    \begin{equation} \label{eqiii}
        E(\sigma t) \le (c \sigma)^{\frac{1}{\alpha}} E(t)
    \end{equation}
    for some constants $c \in (0,1)$ and $\sigma >1$,  and for sufficiently large $t$. Taking the Young conjugate of both sides, and making use of \eqref{YoungCompl} tell us that \eqref{eqiii} is in turn equivalent to
    \begin{equation}\label{40a}
        \widetilde{E}\bigl(t\sigma^{-1} \bigr)
            \ge (c \sigma)^{\frac{1}{\alpha}} \widetilde{E}\bigl( t(c\sigma)^{-\frac{1}{\alpha}} \bigr)
    \end{equation}
    for large $t$. Setting $\varrho = c^\frac{1}{\alpha}\sigma^{\frac{1}{\alpha}-1}$, and changing variables, equation \eqref{40a} reads
   \begin{equation*}
        \widetilde{E}(\varrho t)
            \ge c^{\frac{1}{\alpha-1}}\varrho^{\frac{1}{1-\alpha}} \widetilde{E}(t)
    \end{equation*}
    for large $t$.
   Thus, it suffices to show that $\varrho>1$.  Combining \eqref{kt} and  \eqref{eqiii} yields
	\begin{equation*}
		\sigma E(t)
        	\le E(\sigma t)
        	\le (c \sigma)^{\frac{1}{\alpha}} E(t)
	\end{equation*}
    for large $t$, whence $\varrho\ge 1$. If $\varrho >1$ we are done. On the other hand,  $\varrho=1$ if and only if $E(t)= t$ for large $t$, and the latter condition implies that  $\widetilde{E}=\infty$ near infinity, so that (iv) is  trivially satisfied.
 	\\ The proof of the reverse implication is similar.
  \par\noindent
\textit{(ii) is equivalent to (iv)}.  This is established  in \citep[Lemma 2.3~(i)]{Stromberg:Indiana}.
 \par\noindent
\textit{(iv) is equivalent to (vi)}. The proof of this fact follows along the same lines as that of the equivalence of  (iii) and (v), and will be omitted, for brevity.
\end{proof}

\def\Binf{B_{\alpha,\beta}^{\infty}}
\def\Beinf#1{{#1}_{\alpha,\beta}^{\infty}}
\def\Ginf{G_{\alpha,\beta}^{\infty}}

\def\B{B_{\alpha,\beta}}
\def\Be#1{{#1}_{\alpha,\beta}}
\def\G{G_{\alpha,\beta}}

\def\NI{\text{near infinity}}
\def\GL{\text{for $t> 0$}}

We next analyze connections between the Boyd indices of a  Young function $B$, and those of the Young function $\B$ defined, for $0<\alpha<1$, $\beta>0$  and
$\alpha+1/\beta\ge 1$, as
		\begin{equation}
			\B(t) = \int_{0}^{t} \frac{\G^{-1}(s)}{s}\d s
				\quad \text{ for $t\ge 0$,}
			\label{Bab}
		\end{equation}
where $\G\colon [0, \infty) \to [0,\infty)$ is given  by
    \begin{equation}\label{dic12}
        G_{\alpha , \beta}(t)
            = \begin{cases}
				tB^{-1}(1) &\text{if $0 \leq t \leq 1$,}
            \\
            	t \displaystyle\inf_{1\le s\le t}
            	B^{-1}\bigl(s^{1/\beta }\bigr) s^{\alpha-1}
                &\text{if $t > 1$.}
			\end{cases} 
	\end{equation} 
Note that, by the same  argument as in Remark~\ref{rem:AG}, $\B$ is actually a Young function,  and
\begin{equation*}
	\B^{-1}(t) \simeq \G(t)
		\quad \text{for $t> 0$.}
\end{equation*}
Let us also observe that
\begin{equation*}
	1\le I_{\B} \le \frac{1}{\alpha}
\end{equation*}
for every $B$.
This follows from the fact that
\begin{equation}\label{dic11}
\B^{-1}(t)\,t^{-\alpha}
	\simeq \inf_{1\le s<\infty} B^{-1}(s)
    	\max\bigl\{ 1,  t/s \bigr\}^{1-\alpha}
		\quad \text{ for $t\geq 1$,}
\end{equation}
and that the right-hand side of \eqref{dic11} is a non-decreasing function.

Under the additional assumption  that 
\begin{equation} \label{Bzero}
	\inf_{0<t<1} \frac {B(t)}{ t^\frac{1}{\beta(1-\alpha)}}>0,
\end{equation}
 we also define
\begin{equation} \label{Babinf}
	\Binf(t) = \int_{0}^{t} \frac{\Ginf{}^{-1}(s)}{s}\d s
		\quad \text{ for $t\ge 0$,}
\end{equation}
where $\Ginf \colon [0, \infty) \to [0,\infty)$ is given  by
\begin{equation} 
	\label{apr1}
	\Ginf (t)
		 = t \inf_{0 < s\le t}
			B^{-1}\bigl(s^{1/\beta }\bigr) s^{\alpha-1}
			\quad\GL.
\end{equation} 
Note that \eqref{Bzero} guarantees that $\Ginf$ is positive on $(0,\infty)$. Furthermore,
 by an   argument similar to that of   Remark~\ref{rem:AG},
$\Binf$ is a Young function,  and
\begin{equation}
	\label{apr2}
	\Binf{}^{-1}(t)
		\simeq \Ginf (t)
		\quad \GL.
\end{equation}

The next lemma tells us that, under a suitable lower bound for the lower Boyd index of $B$, the infimum on the right-hand side of equations \eqref{dic12} and \eqref{apr1} can be disregarded.

\begin{lemma} \label{lemm:infout} Let $B$ be  a Young function, and let $0<\alpha<1$,  $\beta>0$  and
$\alpha+1/\beta\ge 1$. 
\\ {\rm(i)} Assume that 
\begin{equation}\label{march2}
	i_B > \frac{1}{\beta(1-\alpha)}.
\end{equation}
Then 
\begin{equation}
\label{march3}
	\inf_{1\le s\le t} B^{-1} \bigl( s^{1/\beta} \bigr)\, s^{\alpha-1}
		\simeq B^{-1} \bigl( t^{1/\beta} \bigr)\, t^{\alpha-1}
		\quad\NI.
\end{equation}
Hence, 
\begin{equation*}
	\B^{-1}(t)
		\simeq B^{-1}\bigl( t^{1/\beta} \bigr)\,t^{\alpha}
		\quad\NI.
\end{equation*}
	Conversely if \eqref{march3} holds, then 
$i_B \ge \textstyle\frac{1}{\beta(1-\alpha)}$.
\\{\rm (ii)} Assume  in addition that \eqref{Bzero} holds. If
\begin{equation}
	\label{eq:lowerindexbig}
	i_B^\infty> \frac{1}{\beta(1-\alpha)},
\end{equation}
then
\begin{equation}
\label{eq:infout}
	\inf_{0< s\le t} B^{-1} \bigl( s^{1/\beta} \bigr)\, s^{\alpha-1}
		\simeq B^{-1} \bigl( t^{1/\beta} \bigr)\, t^{\alpha-1}
		\quad\GL.
\end{equation}
Hence, 
\begin{equation}\label{dic14}
	{\Binf}^{-1}(t)
		\simeq B^{-1}\bigl( t^{1/\beta} \bigr)\,t^{\alpha}
		\quad\GL.
\end{equation}
	Conversely if \eqref{eq:infout} holds, then 
$i_B^\infty \ge \textstyle\frac{1}{\beta(1-\alpha)}$.
\end{lemma}

\begin{proof}We limit ourselves to proving Part (ii). The proof of Part (i) requires minor modifications.
If $B$ is infinite for large values of its argument, then the its generalized inverse $B^{-1}$ is constant near infinity, and equation \eqref{eq:infout} holds trivially. 
\\ In the remaining part of this proof,  we may thus assume that the function $B$ is finite-valued.
 Equation \eqref{eq:infout} is equivalent to
	\begin{equation}
		\inf_{0< s\le t} \widetilde{B}(s)\, s^{\frac{1}{\beta(1-\alpha )-1}}
			\simeq \widetilde{B}(t)\, t^{\frac{1}{\beta(1-\alpha )-1}}
			\quad\GL.
		\label{eq:infouttildaB}
	\end{equation}
	Indeed, owing to \eqref{eq:YoungCompl}, condition \eqref{eq:infout} is equivalent to
	\begin{equation}
		\inf_{0< s\le t} \frac{s^{\beta(\alpha-1)+1}}{\widetilde{B}^{-1} (s)}\, 
			\simeq \frac{t^{\beta(\alpha-1)+1}}{\widetilde{B}^{-1} (t)}
            \quad\GL,
			\label{eq:infout2}
	\end{equation}
    and equation \eqref{eq:infout2} is in turn equivalent to \eqref{eq:infouttildaB}. 
   On the other hand,  by Proposition~\ref{prop:Bconditions}, condition \eqref{eq:lowerindexbig} is equivalent to 
	\begin{equation}
		I_{\widetilde{B}}^\infty < {\eta}\,,
			\label{eq:tildaindexcond}
	\end{equation}
	where we have set  ${\eta}=\frac 1{\beta(\alpha-1)+1}$. The same proposition ensures that condition  \eqref{eq:tildaindexcond} is  equivalent to the inequality
	\begin{equation}
		\int_{t}^{\infty} \widetilde{B}(s)\,s^{-{\eta}-1}\,\d s
			\le \widetilde{B}(kt)\,t^{-{\eta}}
				\quad\GL,
		\label{eq:inttildacond}
	\end{equation}
	for some constant $k>1$.
    Hence it suffices to show that \eqref{eq:inttildacond} implies \eqref{eq:infouttildaB}. To this purpose, denote by $\rho \in [0, t]$ a number satisfying
	\begin{equation*}
		\inf_{0< s\le t} \widetilde{B}(s)\, s^{-{\eta}}
			= \widetilde{B}(\rho t)\, (\rho t)^{-{\eta}} 
			\quad\GL.
	\end{equation*}
    By the same argument as in the derivation (iii) from    (i)   in
	Proposition~\ref{prop:Bconditions},   one has that
	\begin{equation*}
			\widetilde{B}(\rho t)\, (\rho t)^{-\eta} k^{\eta}
					\ge \int_{\rho t/ k}^{\infty} \widetilde{B}(s)\,s^{-\eta-1} \d s
					\ge \int_{\rho t}^{t} \widetilde{B}(s)\,s^{-\eta-1} \d s
					\ge \widetilde{B}(\rho t)\, (\rho t)^{-\eta}  \log\frac{1}{\rho}
					\quad\GL,
	\end{equation*}
 whence
	$k^{\eta} \ge \log\frac{1}{\rho}$, and
	\begin{equation*}
		\rho \ge e^{-k^\eta} > 0
			\quad\GL.
	\end{equation*}
	In the  proof of Proposition~\ref{prop:Bconditions} it is also shown that $\widetilde{B}$ satisfies the $\Delta_2$-condition. 
	Hence, there exists  a positive constant $c$ such that
	\begin{equation*}
		\widetilde{B}(\rho t )
			\ge \widetilde{B}\bigl(t e^{-k\eta}\bigr)
			\ge c \widetilde{B}(t)
			\quad\GL.
	\end{equation*}
Consequently,
	\begin{equation*}
		\inf_{0< s\le t} \widetilde{B}(s)\, s^{-\eta}
			= \widetilde{B}(\rho t)\, (\rho t)^{-\eta}
			\ge c\rho^{-\eta} \widetilde{B}(t)\, t^{-\eta}
			\quad\GL,
	\end{equation*}
	 whence \eqref{eq:infouttildaB} follows. 
\\
Finally, if \eqref{eq:infout} is in force, then $B^{-1}(t)\,t^{\beta(\alpha-1)}$ is equivalent to a non-increasing function, and therefore $i_B^\infty \ge \textstyle\frac{1}{\beta(1-\alpha)}$.
\end{proof}

We conclude this section by showing that, under assumption \eqref{march2} or \eqref{eq:lowerindexbig},  the upper Boyd indices of $B$ and $\B$,  or of $B$ and  $\Binf$ are determined by each other.
In what follows,  we adopt  the convention that $\frac{1}{\infty}=0$.

\begin{lemma}
\label{lemm:indices}
Let $B$ be a Young function, and let $\alpha$ and $\beta$ be as in Lemma~\ref{lemm:infout}.
\\{\rm(i)} Assume that condition \eqref{march2} holds. 
Then 
\begin{equation*}
	\frac{1}{I_{\B}} = \alpha + \frac{1}{\beta I_{B}}.
\end{equation*}
In particular, $I_{\B} < 1/\alpha$ if and only if $I_B < \infty$.
\\{\rm(ii)} Assume, in addition, that $B$ satisfies condition \eqref{Bzero}. If
\eqref{eq:lowerindexbig} holds,
then 
\begin{equation}\label{dic15}
	\frac{1}{I_{\Binf}^\infty} = \alpha + \frac{1}{\beta I_{B}^\infty}.
\end{equation}
In particular, $I_{\Binf}^\infty < 1/\alpha$ if and only if $I_B^\infty < \infty$.
\end{lemma}

\begin{proof} As in Lemma \ref{lemm:infout}, we only prove Part (ii). By Lemma \ref{lemm:infout}, assumption \eqref{eq:lowerindexbig} implies equation
 \eqref{dic14}. Thereby,
\begin{align*}
    h_{\Binf} ^\infty(t)
        & \simeq \sup_{s>0}
            \frac{{\Binf}^{-1}(st)}
                {{\Binf}^{-1}(s)}
            \simeq t\, \sup_{s>0}
            \frac{ B^{-1} \bigl( (st)^{1/\beta} \bigr)\, (st)^{\alpha-1}}
                { B^{-1} \bigl( s^{1/\beta} \bigr)\, s^{\alpha-1}}
            \simeq t^\alpha\, h_B^\infty \bigl( t^{1/\beta} \bigr) \quad \hbox{for $t>0$.}
\end{align*}
Equation \eqref{dic15} is therefore a consequence of the definition of global upper Boyd index.
\end{proof}

\section{Proof of the main results} \label{proofs main}

A key step in the proof of our main results is the solution of the optimal Orlicz domain space $L^A(0,1)$ for the boundedness of 
 the Hardy type operator  $H_{\alpha, \beta}$ in 
\begin{equation}\label{boundH}
	H_{\alpha, \beta}\colon L^A(0,1) \to L^B(0,1)\,,
\end{equation}
for a given space $L^B(0,1)$. 
Indeed, its boundedness properties characterize, via appropriate reduction principles, the Sobolev type embeddings considered in the present paper. This is the objective of  the following lemma.

\begin{lemma}\label{lemm:OptimalityH}
Let $0<\alpha<1$, $\beta>0$, and $\alpha+1/\beta\ge 1$. Suppose that $B$ is a Young function and let $\B$ be the Young function defined by \eqref{Bab}. If
\begin{equation}
    I_{\B} < \frac{1}{\alpha},
        \label{eq:indexcond}
\end{equation}
then
\begin{equation}
    H_{\alpha, \beta}\colon L^{\B}(0,1) \to L^B(0,1)\,,
\label{eq:HLBDtoLB}
\end{equation}
and $L^{\B}(0,1)$ is the optimal (i.e.\ largest) Orlicz domain space that renders  \eqref{eq:HLBDtoLB} true.
\\
Conversely, if \eqref{eq:indexcond} is not satisfied, then no optimal Orlicz domain space exists in \eqref{eq:HLBDtoLB}, in the sense that any Orlicz space  $L^A(0,1)$ which makes  \eqref{boundH} true can be replaced with a strictly larger Orlicz space from which the operator $H_{\alpha, \beta}$  is still bounded into $L^B(0,1)$.
\end{lemma}

A proof of Lemma~\ref{lemm:OptimalityH} in turn combines  \cite[Theorem B]{Musil}, dealing with weak type estimates for the operator $H_{\alpha , \beta}$ in Orlicz spaces, with a result, contained in Lemma \ref{lemm:HOrlOrl} below  and its Corollary \ref{cor:HOrlOrl},   showing that any weak type estimate for  $H_{\alpha , \beta}$ in Orlicz spaces  is  equivalent to  a corresponding  strong  estimate.  In fact, we also need a variant of \cite[Theorem B]{Musil}, which is the object of the next proposition, where the Hardy-type operator $H_{\alpha, \beta}$ is replaced with the operator $H_{\alpha, \beta}^\infty$, acting on spaces defined in the entire half-line $(0, \infty)$, and defined as 
\begin{equation}\label{Habinf}
    H_{\alpha, \beta}^\infty f (s)
        = \int_{s^\beta}^{\infty} f(r)\,r^{\alpha - 1}\,\d r
        \quad \text{for $s>0$,}
\end{equation}
for any function $f\in \mathcal M(0,\infty)$ whenever the integral in \eqref{Habinf} is defined. Importantly, it is also necessary  to keep track of the dependence of the  constants in the inequalities involving $H_{\alpha, \beta}^\infty$.

\def\LA{{L^A(0,\infty)}}
\def\MB{{M^B(0,\infty)}}
\begin{proposition} \label{prop:GWE}
Let $0<\alpha<1$, $\beta>0$ and $\alpha+1/\beta\ge 1$,
and let $A$ and $B$ be Young functions. Assume 
 that $B$ fulfills condition \eqref{Bzero}.
Then the following two assertions are equivalent.
\begin{enumerate}[\rm (i)]
\item There exists a constant $C_1>0$ such that
\begin{equation} \label{eq:HLAtoMEinf}
	\| H_{\alpha, \beta}^\infty f \|_\MB
		\le C_1 \|f\|_\LA
\end{equation}
for every $f\in \LA$.
\item
There exists a constant $C_2$ such that
\begin{equation} \label{intcondinf}
	\int_0^t \frac{\widetilde{A}(s)}{s^{{1/(1-\alpha)}+1}}\,\d s
		\le \frac{\widetilde{\Binf}(C_2t)}{t^{1/(1-\alpha)}}
			\quad \text{for $t>0$,}
\end{equation}
where the $\Binf$ is the Young function defined in \eqref{Babinf}.
\end{enumerate}
Moreover the constants $C_1$ and $C_2$ only depend on each other and on $\alpha$.
\end{proposition}

\begin{proof}
\def\LTA{{L^{\widetilde{A}}(0,\infty)}}
\def\MBD{{(M^B)'(0,\infty)}}
A duality argument (see e.g. \cite[Lemma 8.1]{CPS}), combined with equation \eqref{holder}, tells us that inequality  \eqref{eq:HLAtoMEinf} is equivalent to
\begin{equation} \label{upr1}
	\biggl\| t^{{\alpha}-1}\int_0^{t^{\frac{1}{\beta}}} f(s)\,\d s \biggr\|_\LTA
		\le C_1 \|f\|_\MBD
\end{equation}
for every $f\in \MBD$, where $\MBD$ is defined as in \eqref{MB'}. We claim that inequality \eqref{upr1} is in turn  equivalent to
\begin{equation} \label{upr2}
	\biggl\| t^{{\alpha}-1}\int_0^{t^{\frac{1}{\beta}}} f^*(s)\,\d s \biggr\|_\LTA
		\le C_1 \|f^*\|_\MBD
\end{equation}
for every $f\in \MBD$. Indeed, the fact that \eqref{upr1} implies \eqref{upr2} is trivial, whereas the reverse implication 
 follows from a basic property of rearrangements \cite[Lemma 2.1, Chapter 2]{BS}.
Next, by \citep[Proposition 3.4]{Musil}, inequality \eqref{upr2} is equivalent to
the same inequality restricted just to characteristic functions of the
sets of finite measure, namely to the inequality  
\begin{equation} \label{upr3}
	\biggl\| t^{{\alpha}-1}\int_0^{t^{\frac{1}{\beta}}} \chi_{(0,r)}(s)\,\d s \biggr\|_\LTA
		\le C_1 \|\chi_{(0,r)}\|_\MBD
		\quad\text{for $r>0$}.
\end{equation}
Owing to the equality
$$ \|\chi_{(0,r)}\|_{M^B(0, \infty)}\,\|\chi_{(0,r)}\|_\MBD =r \quad \hbox{for $r>0$}$$
(see \cite[Theorem 5.2, Chapter 2]{BS}), and to equation
\eqref{phiM} with $A$ replaced by $B$, 
\begin{equation} \label{upr4}
	\|\chi_{(0,r)}\|_\MBD
	 \simeq r\, B^{-1}(1/r)  \quad \hbox{for $r>0$,}
\end{equation}
up to absolute equivalence constants.
On the other hand, computations show that
\begin{equation}\label{april1}
	\biggl\| t^{{\alpha}-1} \int_0^{t^\frac{1}{\beta}} \chi_{(0,r)}(s)\,\d s \biggr\|_\LTA
		\simeq  r\, \|t^{{\alpha}-1} \chi_{(r^\beta,\infty)}(t) \|_\LTA
			\quad\text{for $r>0$,}
\end{equation}
up to equivalence constants depending on $\alpha$. 
The right-hand side of \eqref{april1} is finite if and only if the integral on the left-hand side of \eqref{intcondinf} converges. Moreover, if this is the case, then
\begin{equation} \label{upr5}
	\|t^{{\alpha}-1} \chi_{(r^\beta,\infty)}(t) \|_\LTA
		= \frac{r^{\beta(\alpha-1)}}{F^{-1}(r^{-\beta})}
		\quad\text{for $r>0$,}
\end{equation}
where $F\colon (0, \infty) \to [0, \infty)$ is the (increasing) function  defined by
\begin{equation*}
	F(t)=\frac{1}{1-\alpha}\, t^{\frac{1}{1-\alpha}}
		\int_0^t \frac{\widetilde{A}(s)}{s^{{1/(1-\alpha)}+1}}\,\d s
			\quad\text{for $t>0$}.
\end{equation*}
Combining \eqref{upr4}, \eqref{upr5} and \eqref{upr3} tells us that  \eqref{upr2}, and hence \eqref{eq:HLAtoMEinf}, is equivalent to
the existence of a constant $C_3$, depending on $\alpha$, such that  
\begin{equation} \label{upr7}
	\frac{1}{F^{-1}(t)}
		\le C_3 B^{-1} (t^{1/\beta})\, t^{\alpha-1}
		\quad\text{for  $t>0$}.
\end{equation}
Since $F$ is increasing, inequality \eqref{upr7} is equivalent to
\begin{equation*}
	\frac{1}{F^{-1}(t)}
		\le C_3 \frac{\Ginf (t)}{t}
		\quad\text{for  $t>0$.}
\end{equation*}
Finally, by equations \eqref{apr2} and \eqref{eq:YoungCompl}, 
inequality  \eqref{eq:HLAtoMEinf} is equivalent to 
\begin{equation}\label{march10}
	\widetilde{\B}^{-1}(t)
		\le C_4 {F^{-1}(t)}
		\quad\text{for  $t>0$,}
\end{equation}
for some constant $C_4$ depending on $\alpha$. Hence, the conclusion follows, on taking inverses of both sides of \eqref{march10}.
\end{proof}

\begin{lemma} 
\label{lemm:HOrlOrl}
Let $\alpha$, $\beta$, $A$ and $B$ be as in Proposition \ref{prop:GWE}.
If
\begin{equation}
    H_{\alpha, \beta}^\infty \colon L^A(0,\infty) \to M^B(0,\infty),
        \label{eq:HLAtoME}
\end{equation}
then
\begin{equation}
    H_{\alpha, \beta}^\infty \colon L^A(0,\infty) \to L^B(0,\infty).
        \label{eq:HLAtoLE2}
\end{equation}
Moreover, the norms of the operator $H_{\alpha, \beta}^\infty$ in \eqref{eq:HLAtoME} and \eqref{eq:HLAtoLE2}  are equivalent, up to multiplicative constants independent of $A$ and $B$.
\end{lemma}

\begin{proof}
 Throughout  this proof, we adopt the abridged notation $H$ for  $H_{\alpha, \beta}^\infty$. 
\\Given $N>0$, define the  Young functions $A_N$ and $B_N$ as 
\begin{equation}\label{ABL}
	A_N(t) = \frac{A(t)}{N}
		\quad\text{and}\quad
	B_N(t) = \frac{1}{N^{{1/\beta}}} B\bigl(tN^{-\alpha}\bigr) \quad \hbox{for $t \geq 0$.}
\end{equation}
We claim that  equation \eqref{eq:HLAtoME} implies that
	\begin{equation}
			H\colon L^{A_N}(0,\infty) \to M^{B_N}(0,\infty)\,,
					\label{eq:HLALtoMEL}
	\end{equation}
	with  operator norm independent of $N$. To prove this claim,  we make use of Proposition~\ref{prop:GWE}, which tells us that  \eqref{eq:HLAtoME}
	is equivalent to the existence of a positive constant $C$ such that
	\begin{equation}
			\int_0^t \frac{\widetilde{A}(s)}{s^{{1/(1-\alpha)}+1}}\,\d s
					\le \frac{\widetilde{\Binf}(Ct)}{t^{1/(1-\alpha)}}
							\quad \text{for $t>0$,}
					\label{eq:redCond}
	\end{equation}
	where $\Binf$ is the Young function defined by \eqref{Babinf}.
One can verify that the function $\Beinf{(B_N)}$, associated with $B_N$ as in \eqref{Babinf}, satisfies
    \begin{equation*}
        \Beinf{(B_N)} = \frac{\Binf}{N}\,,
    \end{equation*}
		and that inequality \eqref{eq:redCond} holds with $A$ and $\Binf$ replaced by
		$A_N$ and $\Beinf{(B_N)}$, respectively, with the same constant $C$. 
 Proposition~\ref{prop:GWE}
		again tells us that \eqref{eq:HLALtoMEL} holds, with operator norm independent
		of $N$.
	\\
    Now, given any  function  $f\in \mathcal M_+ (0,\infty)$ such that
     \begin{equation}
        0<\int_{0}^{\infty} A\bigl( f(r) \bigr)\,\d r \le 1\,,
            \label{eq:Lle1}
    \end{equation}
    set
  \begin{equation*}
        N = \int_{0}^{\infty} A\bigl( f(r) \bigr)\,\d r.
    \end{equation*}
    Thanks to \eqref{eq:HLALtoMEL}, we have that
    \begin{equation}\label{dic3}
        \| Hf \|_{M^{B_N}(0,\infty)}
            \le C \|f\|_{L^{A_N}(0,\infty)}
            \le C\,,
    \end{equation}
    for some constant $C$ independent of $N$ and $f$, since, by the very definition of Luxemburg norm in Orlicz spaces,
    \begin{equation}\label{dic4}
        \|f\|_{L^{A_N}(0,\infty)} \leq 1.
    \end{equation}
   Equations \eqref{dic3}--\eqref{dic4},  inequality \eqref{f*f**} and  equation \eqref{ABL} tell us
    \begin{equation*}
         C\geq \| Hf \|_{M^{B_N}(0,\infty)}
            \ge \sup_{0<t<\infty} \frac{t}{B_N^{-1 } \bigl({\textstyle \frac{1}{|\{ Hf > t \} |}} \bigr)}
            = \sup_{0<t<\infty} \frac{t}{N^{\alpha} B^{-1}\Bigl({\textstyle \frac{N^{1/\beta}}{|\{ Hf > t \} |}} \Bigr)} \quad \hbox{for $t>0$,}
    \end{equation*}
    namely
    \begin{equation}
        | \{ Hf > t \}|
            \,B \Biggl( \frac{t}{C \bigl( \int_{0}^{\infty} A\bigl( f(r) \bigr)\,\d r\bigr)^{\alpha}} \Biggr)
                \leq \biggl( \int_{0}^{\infty} A\bigl(f(r)\bigr)\,\d r \biggr)^{\frac1\beta}
            \quad \hbox{for $t>0$.}
        \label{eq:weakIneq}
    \end{equation}
   From inequality \eqref{eq:weakIneq}, via assumption  \eqref{eq:Lle1} and property \eqref{kt} applied to $B$, one can deduce that
    \begin{equation}
        | \{ Hf > t \}| \,B \Bigl( \frac tC \Bigr)
            \leq \biggl( \int_{0}^{\infty} A\bigl(f(r)\bigr)\,\d r \biggr)^{\alpha+\frac1\beta}
            \quad \hbox{for $t>0$.}
        \label{eq:weakIneq2}
    \end{equation}
    Clearly, inequality \eqref{eq:weakIneq2} continues to hold even if  the integral on the right-hand side vanishes.
   \\ Our next task is  to derive a strong type
    inequality from the weak type inequality \eqref{eq:weakIneq2}. This will be accomplished via  a discretization argument.
    If the (nonnegative) function $Hf$ is unbounded, denote by $\{ s_k \}_{k\in \mathbb Z}$  a~sequence in $(0,\infty)$ such
    that
    \begin{equation}
        Hf (s_k)=2^k 
        \quad \hbox{for $k\in\Z$.}
        \label{24}
    \end{equation}
    In the case when $Hf$ is bounded, we define  the sequence $\{ s_k \}$ similarly, save that now the index $k$ ranges from $-\infty$ to the smallest   $K \in \mathbb Z$  such  that $Hf (0)\leq 2^K$. We then set  $s_K=0$, and define $s_k$ again by 
    \eqref{24}  for $k \leq K-1$.  In what follows, we shall treat these two cases simultaneously, and  $K$ will denote either $\infty$, or an integer, according to whether $Hf$ is unbounded or bounded, respectively. 
\\
Notice that $s_k$ is non-increasing, since  $Hf$ is non-increasing.
    Define
    \begin{equation*}
        f_k=f\chi_{ {\textstyle\lbrack} s_k^\beta,s_{k-1}^\beta {\textstyle )}} \quad \hbox{for $k <K$.}
    \end{equation*}
    If $k<K$, then
    \begin{equation*}
        Hf(s)\leq Hf(s_k) = 2^k
        \quad\text{for}\ s\in(s_k,s_{k-1})\,.
    \end{equation*}
    Hence,
    \begin{align}\label{eq:25}
        \int_{0}^{\infty} B
            \biggl(
                \frac{Hf(s)} {4C }
            \biggr)\,\d s
            & = \sum_{k<K} \int_{s_{k+1}}^{s_{k}} B
            \biggl(
                \frac{Hf(s)} {4C }
            \biggr)\,\d s
            \\
            &\leq \sum_{k<K}\int_{s_{k+1}}^{s_{k}} B
            \biggl(
                \frac{ 2^{k+1} }{4C}
            \biggr)\,\d s
           = \sum_{k<K} (s_{k}-s_{k+1}) B
            \biggl(
                \frac{ 2^{k-1} } {C}
            \biggr).
           \nonumber
    \end{align}
    Given any $k<K$, 
    \begin{align*}
        Hf_{k}(s)
            &\geq \int_{s_{k}^\beta}^{\infty}
                            f_{k}(r)\, r^{\alpha-1}\,\d r
            =     \int_{s_{k}^\beta}^{\infty}
                            f(r)\chi_{\bigl[ s_{k}^\beta,s_{k-1}^\beta \bigr)}(r)\, r^{\alpha-1}\,\d r   =     \int_{s_{k}^\beta}^{s_{k-1}^\beta}
                            f(r)\, r^{\alpha-1}\,\d r
                \\
            & = Hf(s_{k}) - Hf(s_{k-1})
                = 2^{k-1} \quad \hbox{for $s\in[s_{k+1},s_{k})$.}
    \end{align*}
    Consequently,
    \begin{equation}
        [s_{k+1},s_{k})
            \subset \bigl\{ Hf_{k} \geq 2^{k-1}\bigr\}\quad \hbox{for $k<K$.}
        \label{eq:inclusion}
    \end{equation}
    From inclusion~\eqref{eq:inclusion} and inequality \eqref{eq:weakIneq2} we obtain that
    \begin{align}\label{eq:28}
        (s_{k}-s_{k+1}) B
            \biggl(
                \frac{ 2^{k-1} }{C}
            \biggr)
         \leq
            \bigl| \bigl\{ Hf_{k} \geq 2^{k-1}\bigr\} \bigl| B
            \biggl(
                \frac{ 2^{k-1} }{C}
            \biggr)
         \leq \biggl( \int_{0}^{\infty}A \bigl( f_{k}(r) \bigr)\,\d r \biggr)^{\alpha+\frac1\beta}
    \end{align}
     for $k<K$.
    Coupling~\eqref{eq:25} with \eqref{eq:28},
    and exploiting  the fact that $\alpha+1/\beta \ge 1$
    yield
    \begin{align}\label{eq:29}
        \int_{0}^{\infty} B
            \biggl(
                \frac{Hf(s)}{4C}
            \biggr)\,\d s
            & \leq \sum_{k<K}
                \biggl(
                    \int_{0}^{\infty} A \bigl( f_{k}(r) \bigr)\,\d r
                \biggr)^{\alpha+\frac1\beta}
			\\ \nonumber  & \leq
                \biggl(
                    \sum_{k<K} \int_{0}^{\infty} A \bigl( f_k(r) \bigr)\,\d r
                \biggr)^{\alpha+\frac1\beta}
            \leq
                \biggl(
                    \int_{0}^{\infty} A \bigl( f(r) \bigr)\,\d r
                \biggr)^{\alpha+\frac1\beta},\nonumber
    \end{align}
    for every  function $f\in \mathcal M_+(0,\infty)$ satisfying the second inequality in \eqref{eq:Lle1}.
    Inequality \eqref{eq:29} implies 
equation \eqref{eq:HLAtoLE2}.    
\end{proof}

\begin{corollary} 
\label{cor:HOrlOrl}
Let $0<\alpha<1$, $\beta>0$ and $\alpha+1/\beta\ge 1$. 
Let $H_{\alpha , \beta}$  be the Hardy type operator defined by \eqref{Hab}. 
Assume  that $A$ and $B$ are Young functions such that
\begin{equation}
    H_{\alpha , \beta}\colon L^A(0,1) \to M^B(0,1).
        \label{eq:HLAtoME01}
\end{equation}
Then
\begin{equation}
    H_{\alpha , \beta} \colon L^A(0,1) \to L^B(0,1).
        \label{eq:HLAtoLE201}
\end{equation}
In particular,  the  space $L^A(0,1)$ is the optimal Orlicz domain  in \eqref{eq:HLAtoME01} if and only if it is  the optimal Orlicz domain  in  \eqref{eq:HLAtoLE201}. 
\end{corollary}

\begin{proof}
Suppose that $A$ and $B$ are Young functions such that 
\eqref{eq:HLAtoME01} holds. 
Let us make some preliminary reduction. To begin with, we may suppose that both $A$ and $B$ are finite-valued, or, equivalently, that neither $L^A(0,1)$ nor $L^B(0,1)$ agrees with $L^\infty (0,1)$. Indeed, if $L^A(0,1)=L^\infty (0,1)$, then \eqref{eq:HLAtoLE201} holds trivially, since
$$H_{\alpha, \beta}\colon L^\infty (0,1) \to L^\infty (0,1),$$
and $L^\infty (0,1) \to L^B(0,1)$ for every Young function $B$. On the other hand, if $L^B(0,1) = L^\infty (0,1)$, then $M^B(0,1) = L^B(0,1)$, and hence \eqref{eq:HLAtoLE201} is nothing but \eqref{eq:HLAtoME01}. \\ Next, we may assume, without loss of generality, that 
\begin{equation}\label{infinite}
	\sup_{1 \leq t <\infty}
		\frac{t^{1-\alpha}}{B^{-1}(t^{1/\beta})}=\infty\,.
\end{equation}
Actually, if the supremum in \eqref{infinite} is finite, then $t^{\frac{1}{\beta (1-\alpha)}}$ dominates $B(t)$ near infinity, and hence 
$$L^{\frac{1}{\beta (1-\alpha)}}(0,1) \to L^B(0,1).$$
Since $L^A(0,1)\to L^1(0,1)$ for every Young function $A$ and 
\begin{equation}\label{march14}
H_{\alpha, \beta}\colon L^1(0,1) \to L^{\frac{1}{\beta (1-\alpha)}}(0,1),
\end{equation}
equation \eqref{eq:HLAtoLE201} holds also in this case.
\\ We may thus assume that the Young functions $A$ and $B$  are finite-valued, and that \eqref{infinite} is in force. Under these assumptions,  \cite[Theorem B]{Musil} tells us that  \eqref{eq:HLAtoME01} implies that 
\begin{equation*}
	\int _1^t \frac{\widetilde A (s)}{s^{1/(1-\alpha) +1}}\d s
		\leq \frac{\widetilde{\B}(Ct)}{t^{1/(1-\alpha)}}
		\quad \text{for $t \geq t_0$,}
\end{equation*}
for some   constants $C>0$ and $t_0>1$. 
Let us denote by $\widehat A$ and $\widehat B$ two Young functions which agree with
$A$ and $B$ near infinity, and are modified near zero in such a way
 that $\widehat B$ satisfies condition \eqref{Bzero}, and condition \eqref{intcondinf} holds with $A$ and $\Binf$ replaced by $\widehat A$ and  $\Beinf{\widehat B}$. Hence, by Proposition \eqref{prop:GWE}, 
\begin{equation*}
    H_{\alpha , \beta}^\infty \colon L^{\widehat A}(0,\infty) \to M^{\widehat B}(0,\infty)\,,
\end{equation*}
thus, by Lemma~\ref{lemm:HOrlOrl}, 
\begin{equation}\label{march12}
     H_{\alpha , \beta}^\infty \colon L^{\widehat A}(0,\infty) \to L^{\widehat B}(0,\infty)\,.
\end{equation}
Equation \eqref{march12} implies \eqref{eq:HLAtoLE201} since $L^{\widehat A}(0,1)=L^A(0,1)$, and
 $L^{\widehat B}(0,1)=L^B(0,1)$, up to equivalent norms.
\end{proof}

\begin{proof}[Proof of Lemma~\ref{lemm:OptimalityH}] Here, we make use of the simplified notation $H$ for the operator $H_{\alpha , \beta}$. 
The argument of Remark~\ref{rem:AG}, applied with $G_n$ replaced by $G_{\alpha , \beta}$, ensures that  $\B$ is actually a Young function, and
	\begin{equation*}
	\B{}^{-1}(t) \simeq t \inf_{1\le s\le t} B^{-1} \bigl( s^{1/\beta} \bigr)\,s^{\alpha-1} 
		\quad \text{for $t\geq 1$.}
	\end{equation*}
	Let us begin by considering the case when 
	\begin{equation}
		\inf_{1\le s<\infty} B^{-1} \bigl( s^{1/\beta} \bigr)\,s^{\alpha-1} > 0.
			\label{eq:Btriv}
	\end{equation}
    Under assumption \eqref{eq:Btriv} the function $B^{-1}$ cannot be constant near infinity, and hence $B$ is certainly finite-valued.
	Also,  the function $\G (t)$ is equivalent to $t$, and hence $\B(t)$ is equivalent to $t$.  Thereby $I_{\B}=1$, and 
    \begin{equation*}
     L^{\B}(0,1) = L^1(0,1),
    \end{equation*}
    up to equivalent norms.  In order to prove \eqref{eq:HLBDtoLB}, it remains 
     to show  that 
     \begin{equation}\label{march4}
     H\colon L^1(0,1) \to L^B(0,1)\,.
     \end{equation}
      To verify \eqref{march4}, note that
 condition \eqref{eq:Btriv} is equivalent to
	\begin{equation*}
		\sup_{1\le t<\infty} \frac{B(t)}{t^{1/\beta(1-\alpha)}} < \infty\,,
	\end{equation*}
	whence
	\begin{equation*}
		L^{\frac{1}{\beta(1-\alpha)}}(0,1) \to L^B(0,1).
	\end{equation*}
	This piece of information, combined  with \eqref{march14}, yields \eqref{march4}.
    Note that the domain space is trivially optimal in \eqref{march4}, since  $L^1(0,1)$ is the largest Orlicz space on $(0,1)$.
 
 Let us next focus on  the case when 
 \eqref{eq:Btriv} fails, namely \eqref{infinite} holds. 
    Thanks to Lemma~\ref{lemm:HOrlOrl},  an optimal 
     Orlicz domain $L^A(0,1)$ exists in 
    \begin{equation}\label{dic10'}
		H\colon L^A(0,1) \to L^B(0,1)
	\end{equation}
    if and only if it exists in
	\begin{equation}\label{dic8}
		H\colon L^A(0,1) \to M^B(0,1)\,.
	\end{equation}
    On the other hand, since we are assuming that \eqref{infinite} is in force, by  \citep[Theorem A]{Musil} an optimal Orlicz domain $L^A(0,1)$ in \eqref{dic8}
	exists 
	if and only if 
	\begin{equation}\label{dic9}
		\int_1^t \frac{\widetilde{\B}(s)}{s^{{1/(1-\alpha)}+1}}\,\d s
			\le \frac{\widetilde{\B}(Ct)}{t^{1/(1-\alpha)}}
				\quad\text{near infinity,}
	\end{equation}
    for some constant $C>1$.
	Condition \eqref{dic9}  is in turn equivalent to \eqref{eq:indexcond}, thanks to Proposition~\ref{prop:Bconditions}. The fact that,   in case of existence,  the optimal Orlicz space $L^A(0,1)$ in \eqref{dic8}, and hence in \eqref{dic10'}, is actually  $L^{\B}(0,1)$ is proved in   \citep[Theorem A]{Musil} again.
The proof is complete.
\end{proof}

The characterizations \eqref{inclusionsob} and \eqref{inclusionsobrn} of the inclusion relations between Orlicz-Sobolev spaces are established in the following proposition.

\begin{proposition}\label{P:inclusions}
Assume that $m , n \in \N$. Let $A$ and $B$ be Young functions.
\begin{enumerate}[\rm (i)]
\item If $\Omega$ is an open set in $\mathbb R^n$ such that $|\Omega|< \infty$, then 
\begin{equation}\label{P:inclusions1}
W^{m,A}(\Omega) \to W^{m,B}(\Omega)
	\quad \text{if and only if $A$ dominates $B$ near infinity,}
\end{equation}
and 
\begin{equation}\label{P:inclusions2}
W^{m,A}_0(\Omega) \to W^{m,B}_0(\Omega)
	\quad \text{if and only if $A$ dominates $B$ near infinity.}
\end{equation}
\item
\begin{equation}\label{P:inclusions3}
W^{m,A}(\rn) \to W^{m,B}(\rn)
	\quad \text{if and only if $A$ dominates $B$ globally.}
\end{equation}
\end{enumerate}
\end{proposition}

\begin{proof} The \lq\lq if" parts of assertions \eqref{P:inclusions1}--\eqref{P:inclusions3} are  straightforward consequences of \eqref{inclusion}.
\\ The reverse implications in  \eqref{P:inclusions1} and \eqref{P:inclusions2} can be verified as follows. Assume that $|\Omega|<\infty$, and
\begin{equation}\label{P:inclusions4}
W^{m,A}(\Omega) \to W^{m,B}(\Omega) \quad \hbox{or} \quad W^{m,A}_0(\Omega) \to W^{m,B}_0(\Omega)\,.
\end{equation}
Suppose, without loss of generality, that $0 \in \Omega$. Let $\delta >0$ be so small that the cube~$Q$ centered at~$0$, whose sides are parallel to the coordinates axes and have length $2\delta$, is contained in~$\Omega$. Given any  function $f\in L^{A}(-\delta , \delta)$, define the function $v\colon Q \to \mathbb R$ as 
\begin{equation*}
v(x) = \int _0^{x_1} \int _0^{s_1} \dots \int _0^{s_{m-1}} f(s_m)\, \d s_m\, \d s_{m-1} \dots \d s_1 \quad \hbox{for $x \in Q$,}
\end{equation*}
where we have adopted the notation $x=(x_1, x_2, \dots , x_n)$. The function $v$ is $m$-times weakly differentiable in~$Q$. Moreover, $\tfrac {\partial ^k v}{\partial x_1^k} \in L^\infty (Q)$ if $1\le k\leq m-1$, $\tfrac {\partial ^m v}{\partial x_1^m}(x)= f(x_1)$ for $x \in Q$, and any other derivative vanishes identically.  Hence,   $v\in W^{m,A}(Q)$.
By \cite[Theorem 4.1]{CianchiRandolfi},  there exists a bounded linear  extension operator $\mathcal E\colon W^{m,A}(Q) \to W^{m,A}(\rn)$. Fix any function $\eta \in C^\infty _0(\Omega)$  such that $\eta = 1$ in~$Q$. Define 
$u\colon \Omega \to \mathbb R$ as 
\begin{equation}\label{trial}
u = \eta  \, \mathcal E (v)\,.
\end{equation}
Then $u \in W^{m,A}(\Omega)$, and, in fact, $u \in W^{m,A}_0(\Omega)$. By either of embeddings \eqref{P:inclusions4},   $u\in W^{m,B}(\Omega)$ as well, and hence $f \in  L^{B}(-\delta , \delta)$. Owing to the arbitrariness of $f$,  this implies that $L^{A}(-\delta , \delta)\subset L^{B}(-\delta , \delta)$, and by \cite[Theorem 1.8, Chapter 1]{BS}, in fact $L^{A}(-\delta , \delta)\to L^{B}(-\delta , \delta)$. Hence, by \eqref{inclusion}, $A$ dominates $B$ near infinity. 
\par
As far as the \lq\lq only if" part of assertion \eqref{P:inclusions3} is concerned, the  choice of trial functions $u$ as in~\eqref{trial} implies that $A$ dominates $B$ near infinity also when $\Omega = \mathbb R^n$. On the other hand, if  embedding~\eqref{P:inclusions3} is in force, then, in particular, 
$$W^{m,A}(\rn) \to L^{B}(\rn),$$
whence $A$ dominates $B$ also near zero, by \eqref{dic21}. Therefore, 
 $A$ dominates $B$ globally.
\end{proof}

We are now in a position to accomplish the proofs of our main results.

\begin{proof}[Proof of Theorem~\ref{thm:W0}]
The fact that an optimal Orlicz domain space in  \eqref{sobineq0} exists if and
only if \eqref{eq:Johnindex} holds, and that, in the affirmative case, it
agrees with $W^{m, B_n}_0(\Omega)$, follows from Lemma~\ref{lemm:OptimalityH},
via  the equivalence of the Sobolev inequality \eqref{sobineq0} and of the Hardy
type inequality \eqref{hardy0}. Property \eqref{inclusionsob} also plays a role here.
\\
The assertion about the validity of equation \eqref{oct2}   is a consequence of Lemma~\ref{lemm:indices}.
\end{proof}

\begin{proof}[Proof of Corollary~\ref{cor:W0}]
Fix $u\in W^{m, B_n^\infty}_0(\Omega)$, and  assume, without loss of generality, that $\int_\Omega B_n^\infty(|\nabla^m u|)\,\d y < \infty$, otherwise \eqref{oct11} is trivially satisfied. By Proposition~\ref{prop:Bconditions},  assumption \eqref{eq:Johnindexglobal} is 
equivalent to the existence of a constant $C_1>0$, such that
\begin{equation} \label{BDind}
	\int_0^t \frac{\widetilde{B_n^\infty}(s)}{s^{{n/(n-m)}+1}}\,\d s
		\le \frac{\widetilde{B_n^\infty}(C_1t)}{t^{n/(n-m)}}
		\quad\text{for $t>0$}.
\end{equation}
Given $N>0$, let  $B_N$ be the Young function defined as
\begin{equation*}
	B_N(t) = \frac{B\bigl( t N^{-\frac{m}{n}} \bigr)}{N}
		\quad \text{for $t\ge 0$.}
\end{equation*}
Then, the Young function $(B_N)_n^\infty$ associated with $B_N$ as in \eqref{Aglob} satisfies
\begin{equation*}
	(B_N)_n^\infty = \frac{B_n^\infty}{N}.
\end{equation*}
One can thus verify that inequality 
\eqref{BDind} continues to hold with $B_n^\infty$ replaced by $(B_N)_n^\infty$, and
with the same constant $C_1$, whatever $N$ is.
Hence, by Proposition~\ref{prop:GWE} and
Lemma~\ref{lemm:HOrlOrl}, 
\begin{equation*}
	\|H_{\frac mn, 1}^\infty f\|_{L^{B_N}(0,\infty)}
		\leq C_2  \|f\|_{L^{(B_N)_n^\infty}(0,\infty)}
\end{equation*}
for every $f\in L^{(B_N)_n^\infty}(0,\infty)$, for some constant 
for some $C_2$ independent of $N$. In particular, 
\begin{equation*}
	\|H_{\frac mn, 1}f\|_{L^{B_N}(0,1)}\leq C_2  \|f\|_{L^{(B_N)_n^\infty}(0,1)}
\end{equation*}
for every $f\in L^{(B_N)_n^\infty}(0,1)$. Therefore, owing to the equivalence of inequalities \eqref{sobineq0} and \eqref{hardy0},
\begin{equation} \label{eq:nov1}
	\|u\|_{L^{B_N}(\Omega)} \le C \|\nabla^m u\|_{L^{(B_N)_n^\infty}(\Omega)}
\end{equation}
for every $u\in W_0^{m,(B_N)_n^\infty}(\Omega)$, where the  constant $C$ is independent of $N$.
On choosing
\begin{equation*}
	N = \int_{\Omega} B_n^\infty(|\nabla^m u|)\,\d y\,,
\end{equation*}
and observing that $\|\nabla^m u\|_{L^{(B_N)_n^\infty}(\Omega)}\le 1$ with this choice of $N$, inequality  \eqref{eq:nov1} yields $\|u\|_{L^{B_N}(\Omega)}\le C$. Therefore
\begin{equation*}
	\int_\Omega B_N\biggl(\frac{|u(x)|}{C} \biggr)\,\d x \le 1\,,
\end{equation*}
whence, by the definition of $B_N$,
\begin{equation*}
	\int_\Omega B\biggl(\frac{|u(x)|}{C N^{m/n}} \biggr)\,\d x \le N\,,
\end{equation*}
namely \eqref{oct11}.
\end{proof}

\begin{proof}[Proof of Theorem~\ref{thm:domainJohn}]
The proof follows along  the same lines as that of Theorem~\ref{thm:W0}. Here, the equivalence 
of   inequalities \eqref{sobineqjohn} and \eqref{hardy0} comes into play.
\end{proof}

\begin{proof}[Proof of Theorem~\ref{thm:domainRn}]
The reduction principle for  inequality \eqref{sobineqrn} is relevant in this proof. Recall that such a principle  asserts that this inequality is equivalent to the simultaneous validity of  inequality \eqref{hardy0} and of property \eqref{dic21}. Now, assume that condition \eqref{eq:Johnindex} holds. Then, by Lemma~\ref{lemm:OptimalityH}, inequality \eqref{hardy0} holds with either $L^A(0,1)=L^{1}(0,1)$, or $L^A(0,1)=L^{B_n}(0,1)$,   according to whether $m \geq n$ or $1\leq m <n$. On the other hand, 
\eqref{dic21} trivially holds by the very definition of $\bar B$ and $\bar{B}_n$. Thus, inequality \eqref{sobineqrn}, with $A=\bar B$ or $A=\bar{B}_n$, holds, and
embedding \eqref{easyrn} or \eqref{embrn}, respectively, follows. Moreover, $W^{m, \bar{B}}(\rn)$, or $W^{m, \bar{B}_n}(\rn)$ is optimal in \eqref{easyrn} or \eqref{embrn}. Indeed, if 
inequality \eqref{sobineqrn} holds
for some Young function $A$, then, by the reduction principle, $A$ has to dominate $B$ near $0$, and inequality \eqref{hardy0} must hold. By the optimality of the domain space $L^1(0,1)$ or $L^{B_n}(0,1)$ in \eqref{hardy0}, the function $A(t)$ has to dominate $t$ or $B_n(t)$ near infinity. Thus, $A$ dominates $\bar B$ or $\bar B_n$ globally, whence $W^{m, A}(\rn) \to W^{m, \bar{B}}(\rn)$, or $W^{m, A}(\rn) \to W^{m, \bar{B}_n}(\rn)$, thus proving the optimality of $W^{m, \bar{B}}(\rn)$, or $W^{m, \bar B_n}(\rn)$. 
\\ Conversely, suppose  that condition \eqref{eq:Johnindex} fails. Then, by Lemma~\ref{lemm:OptimalityH}, there does not exist an optimal Orlicz space $L^A(0,1)$ in 
\eqref{hardy0}. As a consequence of the reduction principle, and of \eqref{inclusion} and \eqref{inclusionsobrn}, there does not exist an optimal domain Orlicz-Sobolev space in embedding \eqref{sobrn}.
\end{proof}

\begin{proof}[Proof of Theorem~\ref{thm:Traces}]
This is a consequence of Lemma~\ref{lemm:OptimalityH} and of  the
reduction principle for Sobolev  embeddings with measure, which asserts the
equivalence of  inequalities \eqref{traceineq}
and \eqref{hardytrace}.
\end{proof}

\begin{proof}[Proof of Corollary~\ref{cor:Traces}]
The proof of inequality \eqref{oct16} relies upon a scaling argument as in the proof of Corollary~\ref{cor:W0}. Here, $B(t)$ has to be replaced by $B_N(t)=  N^{-\frac \gamma n}B(tN^{-\frac mn})$,
where
\begin{equation*}
	N = \sum_{k=0}^m \int_\Omega B_\gamma(|\nabla^k u|)\,\d y\,.
\end{equation*}
\end{proof}


\begin{thebibliography}{10}

\bibitem{AcerbiMingione}
E.~Acerbi and R.~Mingione.
\newblock Regularity results for stationary electro-rheological fluids.
\newblock {\em Arch. Rat. Mech. Anal.}, 164:213--259, 2002.

\bibitem{Adams}
R.~A. Adams.
\newblock On the {O}rlicz-{S}obolev imbedding theorem.
\newblock {\em J. Funct. Anal.}, 24(3):241--257, 1977.

\bibitem{ACPS}
A.~Alberico, A.~Cianchi, L.~Pick, and L.~Slav{\'{\i}}kov{\'a}.
\newblock Sharp {S}obolev type embeddings in the entire {E}uclidean space.
\newblock {\em Preprint}.

\bibitem{Ball}
J.~M. Ball.
\newblock Convexity conditions and existence theorems in nonlinear elasticity.
\newblock {\em Arch. Rational Mech. Anal.}, 63:337--403, 1976/77.

\bibitem{Baroni}
P.~Baroni.
\newblock Riesz potential estimates for a general class of quasilinear
  equations.
\newblock {\em Calc. Var. Part. Diff. Equat.}, 53:803--846, 2015.

\bibitem{BS}
C.~Bennett and R.~Sharpley.
\newblock {\em Interpolation of operators}, volume 129 of {\em Pure and Applied
  Mathematics}.
\newblock Academic Press, Inc., Boston, MA, 1988.

\bibitem{Boyd:Pacific}
D.~W. Boyd.
\newblock Indices for the {O}rlicz spaces.
\newblock {\em Pacific J. Math.}, 38:315--323, 1971.

\bibitem{BreitSchirra}
D.~Breit and O.~D. Schirra.
\newblock Korn-type inequalities in {O}rlicz-{S}obolev spaces involving the
  trace-free part of the symmetric gradient and applications to regularity
  theory.
\newblock {\em J. Anal. Appl. (ZAA)}, 31:335--356, 2012.

\bibitem{BreitStrofVerde}
D.~Breit, B.~Stroffolini, and A.~Verde.
\newblock A general regularity theorem for functionals with $\varphi$-growth.
\newblock {\em J. Math. Anal. Appl.}, 383:226--233, 2011.

\bibitem{BulDiSchw}
M.~Bul{\'{\i}}{\v{c}}ek, K.~Diening, and S.~Schwarzacher.
\newblock Existence, uniqueness and optimal regularity results for very weak
  solutions to nonlinear elliptic systems.
\newblock {\em Anal. PDE}, 9:1115--1151, 2016.

\bibitem{BulMaMa}
M.~Bul{\'{\i}}{\v{c}}ek, M.~Majdoub, and J.~M{\'{a}}lek.
\newblock Unsteady flows of fluids with pressure dependent viscosity in
  unbounded domains.
\newblock {\em Nonlinear Anal. Real World Appl.}, 11:3968--3983, 2010.

\bibitem{Cianchi:Indiana}
A.~Cianchi.
\newblock A sharp embedding theorem for {O}rlicz-{S}obolev spaces.
\newblock {\em Indiana Univ. Math. J.}, 45(1):39--65, 1996.

\bibitem{Cianchi:Comm}
A.~Cianchi.
\newblock Boundedness of solutions to variational problems under general growth
  conditions.
\newblock {\em Comm. Partial Differential Equations}, 22(9-10):1629--1646,
  1997.

\bibitem{Cianchi:Forum}
A.~Cianchi.
\newblock Higher-order {S}obolev and {P}oincar\'e inequalities in {O}rlicz
  spaces.
\newblock {\em Forum Math.}, 18(5):745--767, 2006.

\bibitem{CKP}
A.~Cianchi, R.~Kerman, and L.~Pick.
\newblock Boundary trace inequalities and rearrangements.
\newblock {\em J. Anal. Math.}, 105:241--265, 2008.

\bibitem{CianchiPick:AM}
A.~Cianchi and L.~Pick.
\newblock Sobolev embeddings into {BMO}, {VMO}, and {$L_\infty$}.
\newblock {\em Ark. Mat.}, 36(2):317--340, 1998.

\bibitem{CianchiPick:TAMS}
A.~Cianchi and L.~Pick.
\newblock Optimal {S}obolev trace embeddings.
\newblock {\em Trans. Amer. Math. Soc.}, 368(12):8349--8382, 2016.

\bibitem{CPStrace}
A.~Cianchi, L.~Pick, and L.~Slav{\'{\i}}kov{\'a}.
\newblock {S}obolev embeddings, rearrangement-invariant spaces and {F}rostman
  measures.
\newblock {\em Preprint}.

\bibitem{CPS}
A.~Cianchi, L.~Pick, and L.~Slav{\'{\i}}kov{\'a}.
\newblock Higher-order {S}obolev embeddings and isoperimetric inequalities.
\newblock {\em Adv. Math.}, 273:568--650, 2015.

\bibitem{CianchiRandolfi}
A.~Cianchi and M.~Randolfi.
\newblock On the modulus of continuity of weakly differentiable functions.
\newblock {\em Indiana Univ. Math. J.}, 60(6):1939--1973, 2011.

\bibitem{DT}
T.~K. Donaldson and N.~S. Trudinger.
\newblock Orlicz-{S}obolev spaces and imbedding theorems.
\newblock {\em J. Funct. Anal.}, 8:52--75, 1971.

\bibitem{EGO}
D.~E. Edmunds, P.~Gurka, and B.~Opic.
\newblock Double exponential integrability of convolution operators in
  generalized {L}orentz-{Z}ygmund spaces.
\newblock {\em Indiana Univ. Math. J.}, 44:19--43, 1995.

\bibitem{Eyring}
H.~J. Eyring.
\newblock Viscosity, plasticity, and diffusion as example of absolute reaction
  rates.
\newblock {\em J. Chemical Physics}, 4:283--291, 1936.

\bibitem{HMT}
J.~A. Hempel, G.~R. Morris, and N.~S. Trudinger.
\newblock On the sharpness of a limiting case of the {S}obolev imbedding
  theorem.
\newblock {\em Bull. Austral. Math. Soc.}, 3:369--373, 1970.

\bibitem{KP}
R.~Kerman and L.~Pick.
\newblock Optimal {S}obolev imbeddings.
\newblock {\em Forum Math.}, 18(4):535--570, 2006.

\bibitem{Korolev}
A.~G. Korolev.
\newblock On the boundedness of generalized solutions of elliptic differential
  equations with nonpower nonlinearities.
\newblock {\em Mat. Sb.}, 180(1):78--100 (Russian), 1989.

\bibitem{NAFSA98}
M.~Krbec and A.~Kufner, editors.
\newblock {\em Nonlinear analysis, function spaces and applications. {V}ol. 6}.
  Academy of Sciences of the Czech Republic, Mathematical Institute, Prague,
  1999.

\bibitem{Lieberman}
G.~M. Lieberman.
\newblock The natural generalization of the natural conditions of
  {L}adyzenskaya and {U}ral'ceva for elliptic equations.
\newblock {\em Comm. Part. Diff. Eq.}, 16:311--361, 1991.

\bibitem{Marcellini}
P.~Marcellini.
\newblock Regularity for elliptic equations with general growth conditions.
\newblock {\em J. Diff. Eq.}, 105:296--333, 1993.

\bibitem{Mazya}
V.~G. Maz'ya.
\newblock {\em Sobolev spaces, with applications to elliptic partial
  differential equations}.
\newblock Springer, Berlin, 2011.

\bibitem{Musil}
V.~Musil.
\newblock Optimal {O}rlicz domains in {S}obolev embeddings into {M}arcinkiewicz
  spaces.
\newblock {\em J. Funct. Anal.}, 270(7):2653--2690, 2016.

\bibitem{Pokhozhaev}
S.~I. Pohozhaev.
\newblock On the imbedding {S}obolev theorem for $pl=n$.
\newblock {\em Doklady Conference, Section Math. Moscow Power Inst.},
  165:158--170 (Russian), 1965.

\bibitem{Strichartz:Indiana}
R.~S. Strichartz.
\newblock A note on {T}rudinger's extension of {S}obolev's inequalities.
\newblock {\em Indiana Univ. Math. J.}, 21:841--842, 1971/72.

\bibitem{Stromberg:Indiana}
J.-O. Str{\"o}mberg.
\newblock Bounded mean oscillation with {O}rlicz norms and duality of {H}ardy
  spaces.
\newblock {\em Indiana Univ. Math. J.}, 28(3):511--544, 1979.

\bibitem{Talenti79}
G.~Talenti.
\newblock Nonlinear elliptic equations, rearrangements of functions and
  {O}rlicz spaces.
\newblock {\em Ann. Mat. Pura Appl.}, 120:159--184, 1979.

\bibitem{Talenti}
G.~Talenti.
\newblock {\em An embedding theorem. Partial differential equations and the
  calculus of variations, Vol. II}.
\newblock Birkh\"auser, Boston, MA, 1989.

\bibitem{Talenti90}
G.~Talenti.
\newblock Boundedness of minimizers.
\newblock {\em Hokkaido Math. J.}, 19:259--279, 1990.

\bibitem{Trudinger}
N.~S. Trudinger.
\newblock On imbeddings into {O}rlicz spaces and some applications.
\newblock {\em J. Math. Mech.}, 17:473--483, 1967.

\bibitem{Wr}
A.~Wri{\'{o}}blewska.
\newblock Steady flow of non-{N}ewtonian fluids--monotonicity methods in
  generalized {O}rlicz spaces.
\newblock {\em Nonlinear Anal.}, 72:4136--4147, 2010.

\bibitem{Yudovich}
V.~I. Yudovich.
\newblock Some estimates connected with integral operators and with solutions
  of elliptic equations.
\newblock {\em Soviet Math. Doklady}, 2:746--749 (Russian), 1961.

\end{thebibliography}

\end{document}